\documentclass[12pt,reqno]{amsart}
\usepackage{graphicx}
\usepackage{amssymb,amsmath}
\usepackage{amsthm}
\usepackage{color}
\usepackage[pdf]{pstricks}
\usepackage{hyperref}
\usepackage{empheq}
\usepackage{pgfplots}

\numberwithin{equation}{section}

\usepackage{amssymb}
\usepackage{epsfig}
\usepackage{extarrows}
\usepackage{amsfonts}
\usepackage{graphicx}
\usepackage{caption}
\usepackage{subcaption}
\usepackage{tikz}
\usepackage{color}

\newcommand{\R}{\mathbb{R}}

\newtheorem{remark}{Remark}

\newtheorem{lemma}{Lemma}
\newtheorem{theorem}{Theorem}

\begin{document}
	
\title[Traveling waves in a local model for shallow water waves]{\bf On smooth and peaked traveling waves in a local model for shallow water waves}

\author{Spencer Locke}
\address[S. Locke]{Department of Mathematics, University of Michigan, Ann Arbor, MI 48109,  USA}
\email{lockespe@umich.edu}
	
\author{Dmitry E. Pelinovsky}
\address[D. E. Pelinovsky]{Department of Mathematics and Statistics, McMaster University, Hamilton, Ontario, Canada, L8S 4K1 and Department of Applied Mathematics, Nizhny Novgorod State Technical University, 24 Minin street, 603950 Nizhny Novgorod, Russia}
\email{pelinod@mcmaster.ca}

\begin{abstract}
We introduce a new model equation for Stokes gravity waves based on conformal transformations of Euler's equations. The local version of the model equation is relevant for dynamics of shallow water waves. It allows us to characterize the traveling periodic waves both in the case of smooth and peaked waves and to solve the existence problem exactly, albeit not in elementary functions. Spectral stability of smooth waves with respect to co-periodic perturbations is proven analytically based on the exact count of eigenvalues in a constrained spectral problem.
\end{abstract}

\maketitle

\footnote{Corresponding author: Dmitry E. Pelinovsky, e-mail: pelinod@mcmaster.ca}

\section{Introduction}

An irrotational motion of incompressible two-dimensional surface water waves can be fully described by means of evolution equations for two canonical variables in one spatial coordinate. This formalism was originated by V. Zakharov \cite{Zakharov68} in the context of stability of traveling periodic waves. 

One approach to develop this formalism systematically is based on the  Dirichlet-to-Neumann (D-N) operator \cite{Craig}. The two nonlinear evolution equations closed with the D-N operator have been studied in many works on water waves, including the recent study of modulational instability of traveling periodic waves \cite{Berti}. See also \cite{CD,Hur,Strauss} for other recent works where three more methods have been explored in the same context. 

Another approach to obtain a closed system of two nonlinear evolution equations for water waves is based on a conformal transformation which maps the fluid domain with a variable surface profile to a fixed rectangular domain. This formalism was introduced 
in \cite{Babenko,Tanvir} and has been explored in the context of traveling periodic waves in \cite{DZ1,DZ2,CC99} and more recently in \cite{Lush1,DS23,DS23SAPM,Lush2,Lush3}. Our work contributes to the analysis of the nonlinear evolution equations obtained in the latter approach. 

The approach based on conformal transformations has been used to tackle many mathematical problems related to water waves such as the existence of standing waves \cite{W1,W2} and bifurcations of quasi-periodic wave solutions from the standing periodic waves \cite{WZ1,WZ2,WZ3}. Holomorphic coordinates were used for analysis of well-posedness of the water wave equations \cite{Ifrim1,Ifrim2}. The particular problem addressed in our work 
is the coexistence of smooth and peaked traveling periodic waves 
for different intervals of wave speeds as well as the linear stability of waves with smooth profiles. 

\subsection{A new model equation}

The purpose of this paper is to introduce a new model equation which shares the same solutions as the traveling wave reduction of Euler's equations in \cite{Babenko} but simplifies the time evolution and, particularly, the linear stability analysis 
of the traveling periodic waves. This model equation can be written in the following nonlocal form,
\begin{equation}
\label{single-eq}
2 c T_h^{-1} \partial_t \eta = (c^2 K_h -1) \eta - \eta K_h \eta - \frac{1}{2} K_h \eta^2,
\end{equation}
where $\eta = \eta(u,t) \in \mathbb{R}$ is the surface elevation in the reference frame moving with the constant wave speed $c > 0$, $t \in \mathbb{R}$ is time, and $u$ is the spatial coordinate defined on the periodic domain $\mathbb{T} = \mathbb{R} \backslash (2\pi \mathbb{Z})$. The spatial coordinate $u$ arises after the conformal transformation of the fluid domain with variable surface elevation $\eta$ to a rectangle $[-h,0] \times [-\pi,\pi]$, where $h > 0$ is the fluid depth.
The linear skew-adjoint operator $T_h^{-1}$ in $L^2(\mathbb{T})$ is defined by the Fourier symbol
\begin{equation}
\label{operator-T}
\widehat{\left( T_h^{-1} \right)}_n = \left\{ \begin{array}{cl} -i \coth(hn), &\quad n \in \mathbb{Z} \backslash \{0\},  \\ 0, & \quad n = 0, \end{array} \right.
\end{equation}
whereas the linear, self-adjoint,  positive  
operator $K_h = T_h^{-1} \partial_u$ in $L^2(\mathbb{T})$ is defined by the Fourier symbol
\begin{equation}
\label{operator-K}
\widehat{\left( K_h \right)}_n = \left\{ \begin{array}{cl} n \coth(hn), &\quad n \in \mathbb{Z} \backslash \{0\},  \\ 0, & \quad n = 0. \end{array} \right.
\end{equation}
Appendix \ref{appendix} explains how the nonlocal evolution equation (\ref{single-eq}) arises in the context of the original Euler's equations. 

Let us obtain the conserved quantities for the nonlocal model (\ref{single-eq}). 
Taking the mean value of (\ref{single-eq}), we get the constraint
\begin{equation}
\label{zero-mean-new}
\oint \eta (1 + K_h \eta) du = 0,
\end{equation}
which represents the zero-mean constraint for the surface elevation $\eta$ in 
the physical spatial coordinate. Furthermore, differentiating (\ref{single-eq}) in $u$, multiplying by $\eta$ and integrating over the period of $\mathbb{T}$ yields 
\begin{align}
c \frac{d}{dt} \oint \eta K_h \eta du &= \oint (c^2 \eta K_h \eta_u - \eta \eta_u - \eta \eta_u K_h \eta - \eta^2 K_h \eta_u - \eta K_h \eta \eta_u ) du 
\notag \\
&= \oint \partial_u \left(\frac{1}{2} c^2 \eta K_h \eta - \frac{1}{2} \eta^2 
- \eta^2 K_h \eta \right) du = 0, \label{momentum-new}
\end{align}
where we have used self-adjointness of $K_h$ in $L^2(\mathbb{T})$ for every solution with $\eta, \eta_u, \eta \eta_u$ in the domain of $K_h$. It follows from (\ref{zero-mean-new}) and (\ref{momentum-new}) that the nonlocal evolution equation (\ref{single-eq}) admits two conserved quantities
\begin{equation}
\label{conserved-quantity}
\oint \eta du \qquad \mbox{\rm and} \qquad \oint \eta K_h \eta du.
\end{equation}
In addition, the evolution equation (\ref{single-eq}) can be written in the 
Hamiltonian form 
\begin{equation}
\label{action}
2c T_h^{-1} \partial_t \eta = \Lambda'_c(\eta), \quad {\rm with} \;\; 
\Lambda_c(\eta) := \frac{1}{2} \oint \left[ c^2 \eta K_h \eta - \eta^2 (1 + K_h \eta) \right] du,
\end{equation}
where $\Lambda_c(\eta)$ is the action related to the conserved energy of the fluid. Critical points of $\Lambda_c$ in the corresponding energy space satisfy the Euler--Lagrange equation 
\begin{equation}
\label{trav-Bab-eq}
(c^2 K_h - 1) \eta = \frac{1}{2} K_h \eta^2 + \eta K_h \eta,
\end{equation}
which is known as Babenko's equation because it 
coincides with the traveling wave reduction of Euler's equations after the conformal transformation \cite{Babenko}. In the context 
of the evolution equation (\ref{single-eq}) with $u$ defined 
in the reference frame moving with the wave speed $c$, solutions of (\ref{trav-Bab-eq}) correspond to the time-independent solutions of (\ref{single-eq}).

\subsection{Local model and main results}

In the deep water limit ($h \to \infty$), we have from (\ref{operator-T}) and (\ref{operator-K}) that
$$
\lim_{h \to \infty} T_h^{-1} = -\mathcal{H} \qquad {\rm and} \qquad  
\lim_{h \to \infty} K_h = -\mathcal{H} \partial_u, 
$$
where $\mathcal{H}$ is the periodic Hilbert transform defined by the Fourier symbol 
\begin{equation}
\label{Hilbert}
\widehat{\mathcal{H}}_n = i \; {\rm sgn}(n), \qquad n \in \mathbb{Z}.
\end{equation}
This work explores the shallow water limit ($h \to 0$), when 
we replace $T_h^{-1}$ and $K_h$ by $-\partial_u$ and $-\partial_u^2$ respectively. In other words, we study herein the local evolution equation 
\begin{equation}
\label{toy-model}
2 c \partial_u \partial_t \eta = (c^2 - 2 \eta) \partial^2_u \eta 
- (\partial_u \eta)^2 + \eta. 
\end{equation}
Appendix \ref{app-B} describes how the local model (\ref{toy-model}) arises 
from $T_h^{-1}$ and $K_h$ as $h \to 0$ and compares it with other phenomenological models for fluid dynamics. 

{\em It is important to emphasize that (\ref{toy-model}) is not 
the asymptotic reduction of (\ref{single-eq}) as $h \to 0$ but rather 
a toy model to understand existence and linear stability of traveling 
periodic waves  in the shallow water limit. }

The local equation (\ref{toy-model}) without the last term was derived in \cite{HS91} in a different (geometric) context and has been referred to as the Hunter--Saxton equation \cite{HZ94}. The same equation (\ref{toy-model}) with the last term was also discussed in \cite{Alber1,Alber2} in connection to the high-frequency limit of the Camassa--Holm equation, one of the toy models for physics of fluids with smooth and peaked waves. Integrability of 
(\ref{toy-model}) was established in \cite{Hone} together with other peaked wave equations such as the reduced Ostrovsky and short-pulse equations. 
Some traveling wave solutions of this and similar equations were studied with 
Hirota's bilinear method in \cite{Matsuno}.

Next we discuss the time evolution and the conserved quantities for the local model (\ref{toy-model}). Taking the mean value of (\ref{toy-model}) for smooth $2\pi$-periodic solutions and integrating by parts yields the constraint 
\begin{equation}
\label{zero-mean-toy}
\oint \left[ \eta + (\partial_u \eta)^2 \right] du = 0, 
\end{equation}
which corresponds to (\ref{zero-mean-new}) also after integration by parts.
Let $\Pi_0 : L^2(\mathbb{T}) \rightarrow L^2(\mathbb{T}) |_{\{1\}^T}$ be a projection operator to the periodic functions with zero mean. The evolution equation (\ref{toy-model}) can be written in the form 
\begin{equation}
\label{local-evolution}
2 c \partial_t \eta = (c^2 - 2 \eta) \partial_u \eta + \Pi_0 \partial_u^{-1} \Pi_0 \left[ (\partial_u \eta)^2 + \eta \right],
\end{equation}
where $\Pi_0 \partial_u^{-1} \Pi_0$ is uniquely defined on the zero-mean functions in $L^2(\mathbb{T})$ with the zero-mean constraint. The evolution equation (\ref{local-evolution}) is a nonlocal version of the inviscid Burgers equation. The initial-value problem for the inviscid Burgers equation is locally well-posed in $H^1_{\rm per}(\mathbb{T}) \cap W^{1,\infty}(\mathbb{T})$. 
Since the mapping 
$$
\Pi_0 \partial_u^{-1} \Pi_0 \left[ (\partial_u \eta)^2 + \eta \right] : H^1_{\rm per}(\mathbb{T}) \cap W^{1,\infty}(\mathbb{T}) \to H^1_{\rm per}(\mathbb{T}) \cap W^{1,\infty}(\mathbb{T})
$$
is bounded on every bounded subset, there exists a unique local solution 
of the evolution equation (\ref{local-evolution}) for every initial data 
in $H^1_{\rm per}(\mathbb{T}) \cap W^{1,\infty}(\mathbb{T})$. 

To get the conserved quantities, we multiply (\ref{toy-model}) by $\partial_u \eta$ and integrate over the period for smooth $2\pi$-periodic solutions $\eta$. This implies the conservation of 
\begin{equation}
\label{Q}
Q(\eta) := \oint (\partial_u \eta)^2 du,
\end{equation}
and in view of the constraint (\ref{zero-mean-toy}), the conservation of 
\begin{equation}
\label{M}
M(\eta) := \oint \eta du.
\end{equation}
The conserved quantities (\ref{Q}) and (\ref{M}) correspond to (\ref{conserved-quantity}). Furthermore, similar to (\ref{action}), we can write (\ref{local-evolution}) in the Hamiltonian form
\begin{equation}
\label{action-toy}
2c \partial_t \eta = - \Pi_0 \partial_u^{-1} \Pi_0 \left[ \frac{1}{2} c^2 Q'(\eta) - \frac{1}{2} H'(\eta) \right],
\end{equation}
where $H$ is the third conserved quantity given by 
\begin{equation}
\label{H}
H(\eta) := \oint \left[ \eta^2 + 2 \eta (\partial_u \eta)^2 \right] du.
\end{equation}

The existence of traveling periodic waves in the local model (\ref{toy-model}) is defined by the second-order equation 
\begin{equation}
\label{ode-model}
(c^2 - 2 \eta) \eta'' - (\eta')^2 + \eta = 0, \qquad u \in \mathbb{T},
\end{equation}
where $\eta = \eta(u)$ is the $2\pi$-periodic wave profile satisfying the constraint (\ref{zero-mean-toy}). 
The linear stability of the traveling wave with the profile $\eta$ 
is defined by the linearized equation 
\begin{equation}
\label{lin-model}
2 c \partial_t \hat{\eta} = -\Pi_0 \partial_u^{-1} \Pi_0 \mathcal{L} \hat{\eta}, \qquad 
\mathcal{L} := -\partial_u (c^2 - 2 \eta) \partial_u - 1 + 2 \eta'', 
\end{equation}
where $\hat{\eta} = \hat{\eta}(u,t)$ is the perturbation to the traveling wave with the profile $\eta = \eta(u)$ satisfying the orthogonality condition 
$\langle 1 - 2 \eta'', \hat{\eta} \rangle = 0$ with the standard inner product in $L^2(\mathbb{T})$. The orthogonality condition $\langle 1 - 2 \eta'', \hat{\eta} \rangle = 0$  follows by expanding the nonlinear constraint (\ref{zero-mean-toy}) near $\eta$. 

The following two theorems describe the main results of this work. Our results are formulated for the single-lobe periodic solutions with an even profile $\eta$ which possesses a single maximum on $\mathbb{T}$ placed at $u = 0$. Such single-lobe periodic solutions are often referred to as Stokes waves.

\begin{theorem}
	\label{theorem-1}
	There exist $c_* := \frac{\pi}{2 \sqrt{2}}$ and $c_{\infty} \in (c_*,\infty)$ such that the stationary equation (\ref{ode-model}) admits a unique single-lobe solution with the profile $\eta \in  C^{\infty}_{\rm per}(\mathbb{T})$ for every $c \in (1,c_*)$ such that 
\begin{equation}
\label{small-amplitude}	
\| \eta \|_{L^{\infty}} \to 0 \quad \mbox{\rm as} \;\; c \to 1
\end{equation} 
and a single-lobe solution with the profile $\eta \in C^0_{\rm per}(\mathbb{T})$ for every $c \in (c_*,c_{\infty})$ satisfying 
\begin{equation}
\label{2-3-singularity}
	\eta(u) = \frac{c^2}{2} - A(c) |u|^{2/3} + \mathcal{O}(|u|^{4/3}) \quad \mbox{\rm as} \;\; u \to 0,
\end{equation}
	for some constant $A(c) > 0$. At $c = c_*$, there exists a unique single-lobe solution with the profile $\eta \in C^0_{\rm per}(\mathbb{T}) \cap W^{1,\infty}(\mathbb{T})$ given explicitly as 
\begin{equation}
\label{quadratic}
	\eta(u) = \frac{1}{16} ( \pi^2 - 4 \pi |u| + 2 u^2), \qquad u \in [-\pi,\pi]
\end{equation}
and extended as a $2\pi$-periodic function on $\mathbb{T}$.
\end{theorem} 

\begin{remark}
	Figure \ref{fig-profiles} shows profiles of the periodic waves of Theorem \ref{theorem-1}. The profiles were obtained numerically by using solutions of the second-order equation (\ref{ode-model}).
\end{remark}
\begin{figure}[htb!]
	\centering
	\includegraphics[width=7cm,height=6cm]{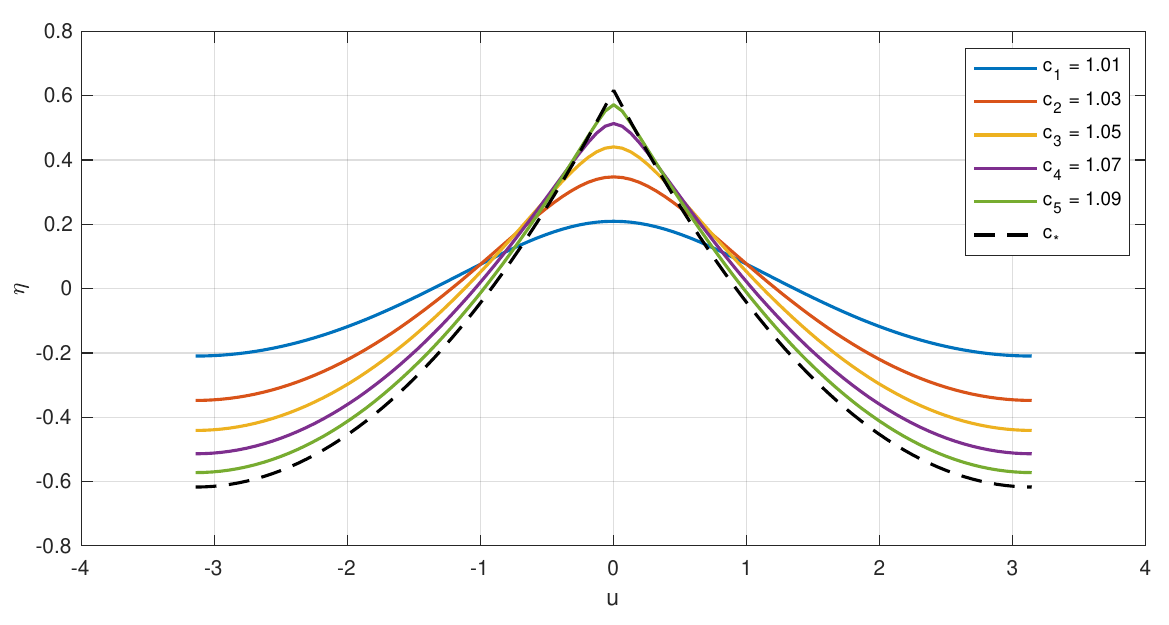} 
		\includegraphics[width=7cm,height=6cm]{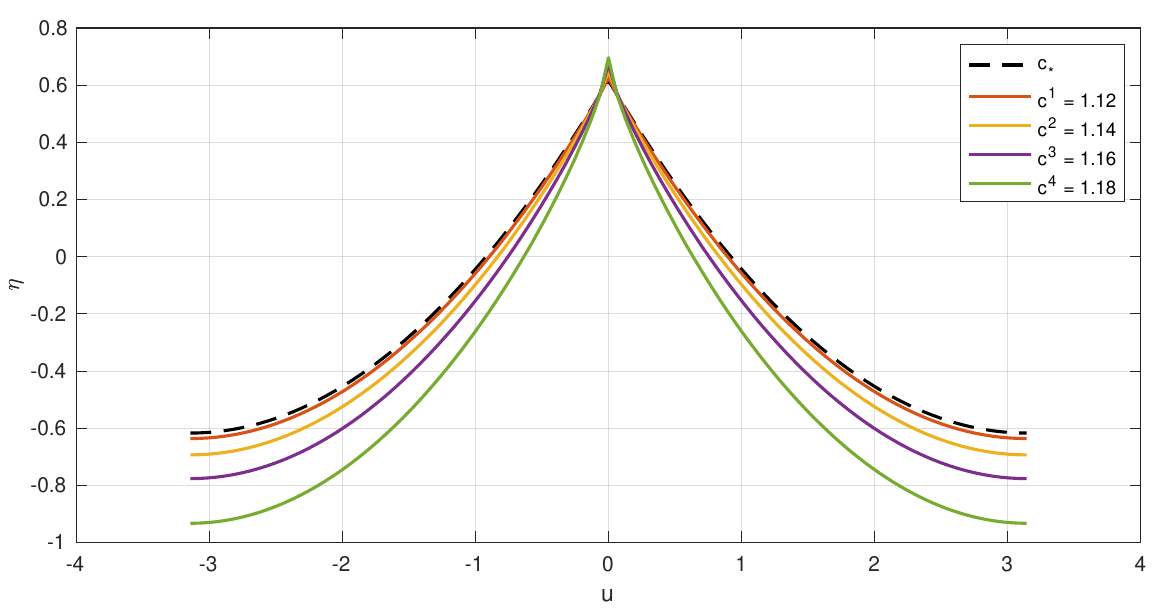} 
	\caption{Profiles in the family of smooth periodic waves (left) and 
	singular periodic waves (right) versus $u$ for different values of wave speeds $c \in (1,c_*)$ and $c \in (c_*,c_{\infty})$, respectively. The dashed line shows the profile of the peaked periodic wave (\ref{quadratic}).}
	\label{fig-profiles}
\end{figure}

\begin{remark}
	There exists another single-lobe solution with the singular behavior (\ref{2-3-singularity}) for every $c \in (0,c_{\infty})$ which is not included in the statement of Theorem \ref{theorem-1} as it does not bifurcate from the zero solution as $c \to 1$ compared to (\ref{small-amplitude}). See Figure \ref{fig-0} for the bifurcation diagram of all single-lobe solutions 
	of (\ref{ode-model}). 
\end{remark}

\begin{figure}[htb!]
	\centering
	\includegraphics[width=12cm,height=10cm]{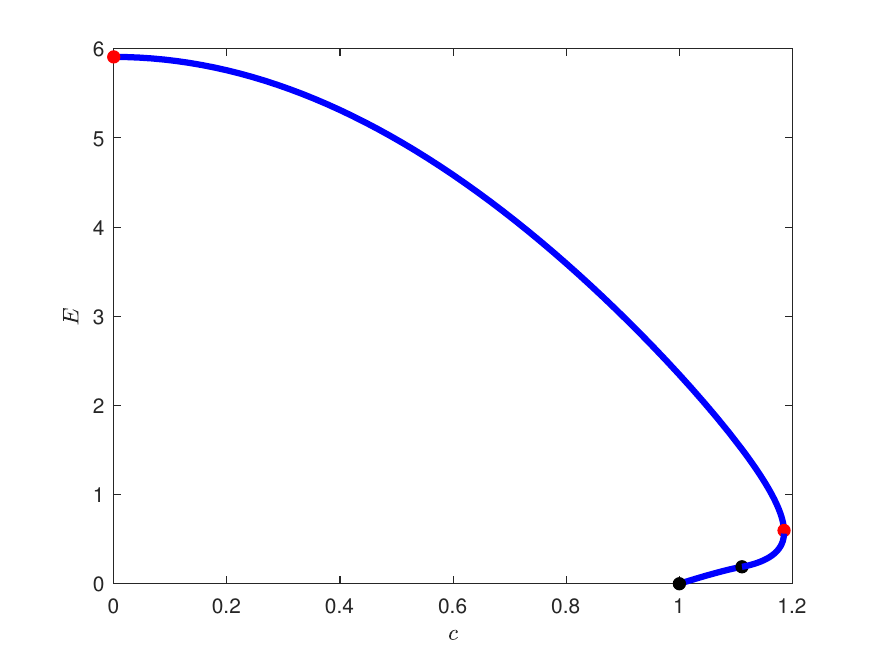} 
	\caption{Dependence of $E := \| \eta \|_{L^{\infty}}$ versus wave speed $c$ for periodic solutions of (\ref{ode-model}). The smooth solutions of Theorem \ref{theorem-1} exist between the black dots. The singular solutions of Theorem \ref{theorem-1} exist between the rightmost black and red dots. The singular solutions which are not included in the statement of Theorem \ref{theorem-1} exist between the red dots.}
	\label{fig-0}
\end{figure}

\begin{remark}
The special solution (\ref{quadratic}) has a peaked profile with a finite jump 
of the first derivative. It is usually referred to as the peaked periodic wave. Such peaked periodic waves are commonly known in other fluid models such as the reduced Ostrovsky equation \cite{BD19,GP1,GP2}, the Camassa--Holm equation \cite{MP-2021}, and the Degasperis--Procesi equation \cite{GP3}.
\end{remark}

\begin{remark}
	The singular behavior (\ref{2-3-singularity}) corresponds to the singularity of the limiting Stokes wave with a $120^o$ angle in the physical coordinate after the conformal transformation. The behavior was rigorously proven for the original Euler's equation in \cite{Amick,Plotnikov,Toland}, with many asymptotic results known in the literature (see \cite{Lushnikov} and references therein). Note that 
	$$
	\max_{u \in \mathbb{T}} \eta(u) = \eta(0) = \frac{c^2}{2}
	$$
	holds for every limiting Stokes wave for which the horizontal velocity 
	at the wave height coincides with the wave speed $c$, see the second equation in system (\ref{euler}) of Appendix \ref{appendix}.
\end{remark}

\begin{theorem}
	\label{theorem-2}
Consider the unique single-lobe solution with the profile $\eta \in C^{\infty}_{\rm per}(\mathbb{T})$ in Theorem \ref{theorem-1} for $c \in (1,c_*)$. 
For every initial data $\hat{\eta}_0 \in H^1_{\rm per}(\mathbb{T})$ satisfying 
\begin{equation}
\label{constraints-linear}
\langle 1, \hat{\eta}_0 \rangle = 0 \quad 
\mbox{\rm and}  \quad \langle \eta'', \hat{\eta}_0 \rangle = 0,
\end{equation}
there exists a unique solution $\hat{\eta} \in C^0(\mathbb{R},H^1_{\rm per}(\mathbb{T}))$ of the linearized equation (\ref{lin-model}) 
with $\hat{\eta} |_{t = 0} = \hat{\eta}_0$ and a unique $a \in C^1(\mathbb{R},\mathbb{R})$ such that 
\begin{equation}
\label{lin-stability}
\| \hat{\eta}(\cdot,t) - a(t) \eta' \|_{H^1_{\rm per}} \leq C \| \hat{\eta}_0 \|_{H^1_{\rm per}}, \quad |a'(t)| \leq C  \| \hat{\eta}_0 \|_{H^1_{\rm per}}, \quad t \in \mathbb{R},
\end{equation}
where $C > 0$ is independent of $\hat{\eta}_0$.
\end{theorem} 

\begin{remark}
Constraints (\ref{constraints-linear}) are preserved in the time 
evolution of the linearized equation (\ref{lin-model}) because they are linearizations of the conserved quantities (\ref{Q}) and (\ref{M}). 
In view of  the constraint $\langle 1 - 2 \eta'', \hat{\eta} \rangle = 0$ imposed on solutions of the linearized equation (\ref{lin-model}), only one constraint 
in (\ref{constraints-linear}) is linearly independent. Imposing 
$\langle \eta'', \hat{\eta}_0 \rangle = 0$ is equivalent to the requirement 
that the perturbation $\hat{\eta}$ does not change the conserved quantity $Q$ in (\ref{Q}) up to the linear approximation. The bound (\ref{lin-stability}) expresses the concept of linear orbital stability of the orbit $\{ \eta(\cdot + \mathfrak{u}) \}_{\mathfrak{u} \in \mathbb{T}}$ of the traveling periodic wave with the profile $\eta$.
\end{remark}

\begin{remark}
The linear orbital stability of Theorem \ref{theorem-2} implies spectral stability 
of the traveling periodic wave in the sense that the spectrum of the associated linearized operator $\partial_u^{-1} \mathcal{L}$ in $L^2(\mathbb{T})$ belongs to $i \mathbb{R}$. It is also worthwhile to point out that an eigenfunction $\hat{\eta}_0$ of the spectral stability problem 
\begin{equation}
\label{spec-stab-KdV}
\partial_u^{-1} \mathcal{L} \hat{\eta}_0 = \lambda_0 \hat{\eta}_0, \qquad \eta_0 \in H^1_{\rm per}(\mathbb{T})
\end{equation}
for every nonzero eigenvalue $\lambda_0 \in \mathbb{C} \backslash \{0\}$ must satisfy the two constraints in (\ref{constraints-linear}). The spectral stability problem in the form (\ref{spec-stab-KdV}) was also considered in \cite{SS18}.
\end{remark}

\begin{remark}
	The proof of Theorem \ref{theorem-2} relies on the construction 
	of the coercive quadratic form $\langle \mathcal{L}\hat{\eta}, \hat{\eta} \rangle$ for $\hat{\eta} \in H^1_{\rm per}(\mathbb{T})$ under the three constraints: 
	$$
	\langle 1, \hat{\eta} \rangle = \langle \eta', \hat{\eta} \rangle = \langle \eta'', \hat{\eta} \rangle = 0.
	$$ 
	The quadratic form is invariant in the time evolution of the linearized equation (\ref{lin-model}). This yields the energetic stability of the traveling periodic wave, ensuring that the periodic wave with the profile $\eta$ is a local minimizer of the energy $H$ subject to fixed $Q$ and $M$ in $H^1_{\rm per}(\mathbb{T})$. If local well-posedness of the nonlinear evolution equation (\ref{local-evolution}) can be shown in $H^1_{\rm per}(\mathbb{T})$, then the energetic stability implies the nonlinear orbital stability of the traveling periodic wave. However, the local well-posedness of (\ref{local-evolution}) holds only in $H^1_{\rm per}(\mathbb{T}) \cap W^{1,\infty}(\mathbb{T})$ and the control of $\| \partial_u \hat{\eta} \|_{L^{\infty}}$ for the perturbation $\hat{\eta}$ does not follow from the conserved quantities (\ref{Q}), (\ref{M}), and (\ref{H}).
\end{remark}

\subsection{Discussion}

The local model (\ref{toy-model}) shows a pattern of the existence and stability of traveling periodic waves parameterized by the wave speed $c$. There is a continuum of wave speeds for the smooth waves with profile $\eta \in C^{\infty}(\mathbb{T})$ bifurcating from the linear limit in (\ref{small-amplitude}) and a continuum of wave speeds for the cusped waves with the profile $\eta \in C^0(\mathbb{T})$ satisfying (\ref{2-3-singularity}). The two continuous families are connected together at a particular value of the wave speed $c = c_*$ for which the wave is peaked with the profile $\eta \in C^0(\mathbb{T}) \cap W^{1,\infty}(\mathbb{T})$. The same phenomenon is observed in the Camassa--Holm equation \cite{GMNP,Len2} and the Degasperis--Procesi equation \cite{GP3,Len1}.

It is rather remarkable that exactly the same $|u|^{2/3}$ singularity in the limiting wave profile with 
$$
\max_{u \in \mathbb{T}} \eta(u) = \eta(0) = \frac{c^2}{2}
$$
is recovered by the local model (\ref{toy-model}) as predicted by the full model for any depth $h$ \cite{Plotnikov}. After the conformal transformation, this singularity yields the limiting Stokes wave with the $120^o$ angle in the physical coordinates. More precise details of such singular profiles are  beyond the current capacities of the asymptotic \cite{Lushnikov} or numerical \cite{Lush1,Lush3} methods. The local model (\ref{toy-model}) gives  precise conclusions that the $|u|^{2/3}$ singularity is obtained in a range of wave speeds $c$ and that the borderline wave profile $\eta \in H^1_{\rm per}(\mathbb{T}) \cap W^{1,\infty}(\mathbb{T})$ has a peaked profile at $c = c_*$ before getting the $|u|^{2/3}$ singularity for $c > c_*$. This might be an artefact of the local model (\ref{toy-model}) since the nonlocal models typically predict only $|u|^{2/3}$ singularity in the fluid models, see \cite{LP24}.

Stability of the traveling periodic waves with singular profiles is a complicated problem, which is out of reach in the current analytical and numerical methods in the nonlocal models \cite{DS23,Lush2}. The local model (\ref{toy-model}) is a promising candidate for showing linear instability of the peaked wave (based on a similar analysis in \cite{GP1,GP2} and \cite{MP-2021}) and for attacking linear instability of the cusped wave with the $|u|^{2/3}$ singularity. 

The remainder of this paper is organized as follows. Section \ref{sec-2} contains the proof of Theorem \ref{theorem-1} on the existence of smooth and peaked periodic waves. Section \ref{sec-3} gives the proof of Theorem \ref{theorem-2} on linear stability of the smooth periodic waves. Appendix \ref{appendix} reviews the Euler's equations after the conformal transformation and discusses how the nonlocal model equation (\ref{single-eq}) arises. Appendix \ref{app-B} describes the local model (\ref{toy-model}) in the context of other phenomenological models for dynamics of fluid surfaces.

\section{Existence of smooth and peaked traveling periodic waves}
\label{sec-2}

Here we consider the single-lobe periodic solutions of the second-order equation (\ref{ode-model}). Recall that a single-lobe periodic solution has the even profile $\eta$ with a single maximum on $\mathbb{T}$ placed at $u = 0$. Theorem \ref{theorem-1} is proven by using the period function for the planar Hamiltonian systems used in a similar context in \cite{GMNP,GP0,GP3,Long}.

We start with the first-order invariant of the second-order equation (\ref{ode-model}) given by the following lemma.

\begin{lemma}
	\label{lemma-1}
	For every solution $\eta \in C^2(a,b)$ of the second-order equation 
	(\ref{ode-model}) with $-\infty \leq a < b \leq \infty$, the following function
	\begin{equation}
	\label{first-order}
E(\eta,\eta') := \frac{1}{2} (c^2 - 2 \eta) (\eta')^2 + \frac{1}{2} \eta^2
	\end{equation}
	is constant for $u \in (a,b)$.
\end{lemma}

\begin{proof}
It is based on the elementary computation:
\begin{align*}
	\frac{d}{du} E(\eta,\eta') &= (c^2 - 2 \eta) \eta' \eta'' - (\eta')^3 + \eta \eta' \\
	&= \eta' \left[ (c^2 - 2 \eta) \eta'' - (\eta')^2 + \eta \right] \\
	&= 0,
\end{align*}
since $\eta \in C^2(a,b)$ satisfies (\ref{ode-model}).
\end{proof}

The next lemma explores the phase portrait for the second-order equation (\ref{ode-model}) on the phase plane $(\eta,\eta')$ obtained from the level curves of $E(\eta,\eta')$ in (\ref{first-order}). We obtain the existence of smooth and singular solutions in terms of the level $\mathcal{E}$ of $E(\eta,\eta')$.

\begin{lemma}
	\label{lemma-2}
For every $c > 0$, there exists $\mathcal{E}_c := \frac{c^4}{8}$ such that every periodic solution to (\ref{ode-model}) with profile $\eta \in C^{\infty}(\mathbb{R})$ belongs to $E(\eta,\eta') = \mathcal{E}$ with $\mathcal{E}\in (0,\mathcal{E}_c)$. For $E(\eta,\eta') = \mathcal{E}_c$, the only solution to (\ref{ode-model}) is the parabola 
\begin{equation}
\label{parabola}
\eta(u) = -\frac{c^2}{2} + \frac{(u-u_0)^2}{8},
\end{equation}
with arbitrary $u_0 \in \mathbb{R}$. For $E(\eta,\eta') = \mathcal{E}$ with $\mathcal{E}\in (\mathcal{E}_c,\infty)$, there exist no bounded solutions to (\ref{ode-model}) with profile $\eta \in C^{\infty}(\mathbb{R})$.
\end{lemma}

\begin{proof}
	The only equilibrium point of (\ref{ode-model}) is the center $(\eta,\eta') = (0,0)$, which corresponds to the minimum of $E(\eta,\eta')$ if $c > 0$. Every periodic solution to (\ref{ode-model}) belongs to the period annulus, which is the largest punctured neighbourhood of the center $(0,0)$ consisting entirely of periodic orbits. 
	
The phase portrait from the level curves of $E(\eta,\eta') = \mathcal{E}$ is shown on Figure \ref{fig-1}. Each level curve defines the profile $\eta$ from integration of 
\begin{equation}
\label{quad}	
	\left(\frac{d \eta}{du} \right)^2 = \frac{2 \mathcal{E} - \eta^2}{c^2 - 2 \eta}.
\end{equation}
The vertical line corresponds to $\eta = \frac{c^2}{2}$ and divides the phase plane into two half planes. For $\eta > \frac{c^2}{2}$, the level curves of $E(\eta,\eta') = \mathcal{E}$ contain no bounded solutions. For $\eta < \frac{c^2}{2}$, bounded level curves exist for $\mathcal{E}\in (0,\mathcal{E}_c)$ with $\mathcal{E}_c := \frac{c^4}{8}$ and contain periodic solutions with profile $\eta \in C^{\infty}(\mathbb{R})$. For $\mathcal{E} = \mathcal{E}_c$, all solutions are given by integrating 
\begin{equation}
\label{quad-peaked}
	\left(\frac{d \eta}{du} \right)^2 = \frac{2 \mathcal{E}_c - \eta^2}{c^2 - 2 \eta} = \frac{1}{4} (c^2 + 2 \eta).
\end{equation}
Differentiating in $u$ gives $\eta''(u) = \frac{1}{4}$. Integrating twice and using (\ref{quad-peaked}) yields (\ref{parabola}) with profile $\eta \in C^{\infty}(\mathbb{R})$. Finally, for $\mathcal{E} > \mathcal{E}_c$, the level curve reaches $\eta = \frac{c^2}{2}$ with the singularity of $\eta'$, which rules out the existence of bounded solutions with profile $\eta \in C^{\infty}(\mathbb{R})$. 	
\end{proof}

	\begin{figure}[htb!]
	\centering
	\includegraphics[width=12cm,height=10cm]{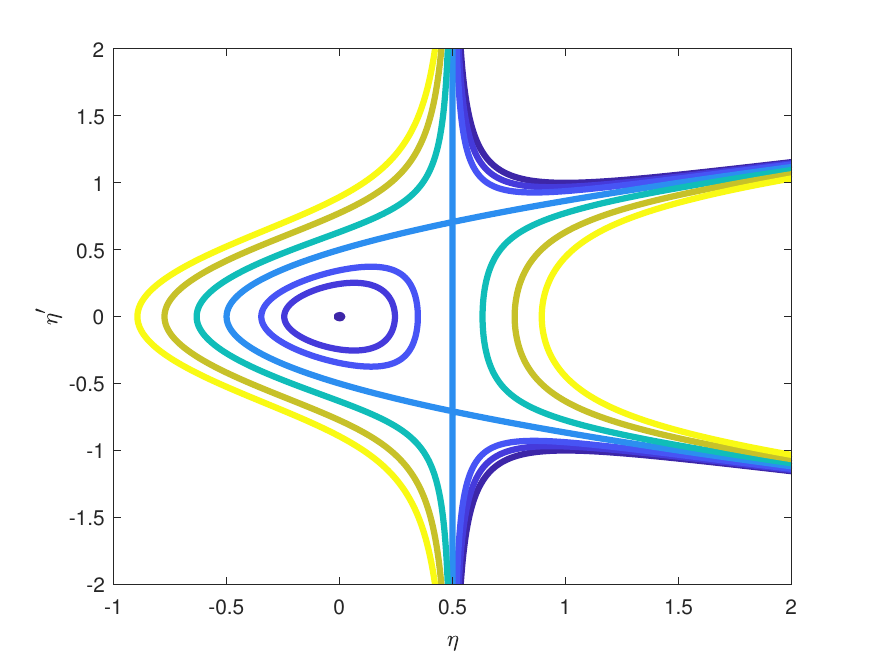}
	\caption{Phase portrait from the level curves of $E(\eta,\eta') = \mathcal{E}$ for $c = 1$.}
	\label{fig-1}
\end{figure}

The next lemma clarifies how the solutions for $\eta < \frac{c^2}{2}$ reach the singularity line $\eta = \frac{c^2}{2}$.

\begin{lemma}
	\label{lemma-3}
	Let $\eta \in C^0(u_-,u_+)$ be a solution for the level curve $E(\eta,\eta') = \mathcal{E}$ for $\mathcal{E} \in (\mathcal{E}_c,\infty)$ such that 
	$\eta(u) \to \pm \frac{c^2}{2}$ as $u \to u_{\pm}$. Then, $-\infty < u_- < u_+ < \infty$ and the solution satisfies 
	\begin{equation}
	\label{sol-expansion}
	\eta(u) = \frac{c^2}{2} - \frac{3^{2/3}}{2^{2/3}} ( \mathcal{E} - \mathcal{E}_c)^{1/3} |u-u_{\pm}|^{2/3} + \mathcal{O}(|u-u_{\pm}|^{4/3}) \qquad \mbox{\rm as} \;\; u \to u_{\pm}.
	\end{equation}
\end{lemma}

\begin{proof}
We consider the level curve $E(\eta,\eta') = \mathcal{E}$ with $\mathcal{E} \in (\mathcal{E}_c,\infty)$ for $\eta < \frac{c^2}{2}$. Then, we have 
\begin{align*}
	\left(\frac{d \eta}{du} \right)^2 &= \frac{2 \mathcal{E} - \eta^2}{c^2 - 2 \eta} \\
	&= \frac{2(\mathcal{E} - \mathcal{E}_c)}{c^2 - 2 \eta} + \frac{1}{2} c^2 + \frac{1}{4} (2 \eta - c^2) \\
	&= \frac{2(\mathcal{E} - \mathcal{E}_c)}{c^2 - 2 \eta} + \mathcal{O}(1) \quad \mbox{\rm as} \;\; \eta \to \frac{c^2}{2}.
\end{align*}
This yields 
$$
\sqrt{c^2 - 2 \eta} 	\left[ 1 + \mathcal{O}(|c^2-2\eta|) \right] \frac{d \eta}{du} = \pm \sqrt{2 (\mathcal{E} - \mathcal{E}_c)} \quad \mbox{\rm as} \;\; u \to u_{\pm} \mp 0.
$$	
Integrating in $u$ yields 
$$
\sqrt{(c^2 - 2 \eta)^3} 	\left[ 1 + \mathcal{O}(|c^2-2\eta|) \right]= \mp 3 \sqrt{2 (\mathcal{E} - \mathcal{E}_c)} (u - u_{\pm}) \quad \mbox{\rm as} \;\; u \to u_{\pm} \mp 0,
$$
which results in the expansion (\ref{sol-expansion}). 
It remains to prove that $[u_-,u_+]$ is compact. This follows from the bounds 
\begin{align*}
u_+ - u_- &= 2 \int_{-\sqrt{2 \mathcal{E}}}^{c^2/2} \frac{\sqrt{c^2 - 2 \eta}}{\sqrt{2 \mathcal{E} - \eta^2}} d \eta \\
&\leq 2 \sqrt{c^2 +2 \sqrt{2 \mathcal{E}}} \int_{-\sqrt{2 \mathcal{E}}}^{c^2/2} \frac{d\eta}{\sqrt{2 \mathcal{E} - \eta^2}} < \infty,
\end{align*}
since the integral in the upper bound is finite.
\end{proof}

In order to analyze bounded periodic solutions in Lemma \ref{lemma-2}, we introduce the following period function associated with (\ref{first-order}) and (\ref{quad}):
\begin{equation}
\label{period-function}
T(\mathcal{E},c) := 2 \int_{-\sqrt{2 \mathcal{E}}}^{\sqrt{2 \mathcal{E}}} \frac{\sqrt{c^2 - 2 \eta}}{\sqrt{2 \mathcal{E} - \eta^2}} d \eta, \qquad \mathcal{E}\in (0,\mathcal{E}_c).
\end{equation}
For the singular solutions in Lemma \ref{lemma-3}, we augment the period function for $\mathcal{E} \geq \mathcal{E}_c$ as
\begin{equation}
\label{period-function-sing}
T(\mathcal{E},c) := 2 \int_{-\sqrt{2 \mathcal{E}}}^{c^2/2} \frac{\sqrt{c^2 - 2 \eta}}{\sqrt{2 \mathcal{E} - \eta^2}} d \eta, \qquad \mathcal{E}\in [\mathcal{E}_c,\infty).
\end{equation}
The next result describes properties of the period function 
$(0,\infty) \ni \mathcal{E} \mapsto T(\mathcal{E},c)$ for every fixed $c > 0$. 

\begin{lemma}
	\label{lemma-4}
	For every $c > 0$, there exist $\mathcal{E}_*, \mathcal{E}_{**} \in (\mathcal{E}_c,\infty)$ that depend on $c$ such that 
\begin{equation}
\label{period-decreasing}
\frac{\partial}{\partial \mathcal{E}} T(\mathcal{E},c) < 0, \quad 
\mathcal{E} \in (0,\mathcal{E}_*)
\end{equation}
and
\begin{equation}
\label{period-increasing}
\frac{\partial}{\partial \mathcal{E}} T(\mathcal{E},c) > 0, \quad 
\mathcal{E} \in (\mathcal{E}_{**},\infty)
\end{equation}
with 
\begin{align*}
T(\mathcal{E},c) &\to 2 \pi c \quad \mbox{\rm as} \;\; \mathcal{E} \to 0, \\
T(\mathcal{E},c) &\to 4 \sqrt{2} c \quad \mbox{\rm as} \;\; \mathcal{E} \to \mathcal{E}_c, 
\end{align*}
and $T(\mathcal{E},c) \to \infty$ as $\mathcal{E}\to \infty$.
In addition, we have 
\begin{equation}
\label{period-speed}
\frac{\partial}{\partial c} T(\mathcal{E},c) > 0, \quad 
\mathcal{E} \in (0,\infty), \quad c \in (0,\infty).
\end{equation}
\end{lemma}

\begin{proof}
	For the smooth periodic solutions, we use the change of variables $\eta = \sqrt{2 \mathcal{E}} x$ in (\ref{period-function}) and obtain 
	$$
T(\mathcal{E},c) = 2 \int_{-1}^{1} \frac{\sqrt{c^2 - 2 \sqrt{2 \mathcal{E}} x}}{\sqrt{1 - x^2}} d x, \qquad \mathcal{E}\in (0,\mathcal{E}_c).
$$	
Since the weak singularity is independent of $\mathcal{E}$, we can differentiate under the integration sign and obtain 
\begin{equation}
\label{tech-eq}
\frac{\partial}{\partial \mathcal{E}} T(\mathcal{E},c) = -\frac{2}{\sqrt{2 \mathcal{E}}} \int_{-1}^{1} \frac{x dx}{\sqrt{1 - x^2} \sqrt{c^2 - 2 \sqrt{2 \mathcal{E}} x}}, \qquad \mathcal{E}\in (0,\mathcal{E}_c).
\end{equation}
The result is strictly negative since 
$$
\frac{|x|}{\sqrt{1 - x^2} \sqrt{c^2 + 2 \sqrt{2 \mathcal{E}} |x|}} < 
\frac{x}{\sqrt{1 - x^2} \sqrt{c^2 - 2 \sqrt{2 \mathcal{E}} x}}, \quad x \in (0,1).
$$
This proves (\ref{period-decreasing}) for $\mathcal{E}\in (0,\mathcal{E}_c)$. 
We also obtain from the same representation
$$
T(\mathcal{E},c) \to 2c \int_{-1}^1 \frac{dx}{\sqrt{1-x^2}} = 2 \pi c \quad \mbox{\rm as} \;\; \mathcal{E} \to 0 
$$
and
$$
T(\mathcal{E},c) \to 2c \int_{-1}^1 \frac{\sqrt{1-x}}{\sqrt{1-x^2}} dx = 4 \sqrt{2} c \quad \mbox{\rm as} \;\; \mathcal{E} \to \mathcal{E}_c. 
$$
Monotonicity (\ref{period-speed}) follows from the positive derivative of (\ref{period-function}) in $c$.

For the singular solutions, we break (\ref{period-function-sing}) into the sum of two terms and use the same change of variables only in the first term:
$$
T(\mathcal{E},c) = 2 \int_{-1}^{-c^2/2\sqrt{2 \mathcal{E}}} \frac{\sqrt{c^2 - 2 \sqrt{2 \mathcal{E}} x}}{\sqrt{1 - x^2}} d x +  2 \int_{-c^2/2}^{c^2/2} \frac{\sqrt{c^2 - 2 \eta}}{\sqrt{2 \mathcal{E} - \eta^2}} d \eta,\qquad \mathcal{E}\in [\mathcal{E}_c,\infty).
$$	
Since $8 \mathcal{E} > c^4$, both terms are differentiable under the integration sign and we obtain 
$$
\frac{\partial}{\partial \mathcal{E}} T(\mathcal{E},c) = -\frac{2}{\sqrt{2 \mathcal{E}}} \int_{-1}^{-c^2/2\sqrt{2 \mathcal{E}}} \frac{x dx}{\sqrt{1 - x^2} \sqrt{c^2 - 2 \sqrt{2 \mathcal{E}} x}} -  2 \int_{-c^2/2}^{c^2/2} \frac{\sqrt{c^2 - 2 \eta}}{\sqrt{(2 \mathcal{E} - \eta^2)^3}} d \eta, 
$$
where the first term is positive and the second term is negative. The first positive term is zero as $\mathcal{E} \to \mathcal{E}_c$, monotonically increasing for $\mathcal{E} \gtrsim \mathcal{E}_c$, and monotonically decreasing as $\mathcal{E} \to \infty$. The second negative term is strictly negative as $\mathcal{E} \to \mathcal{E}_c$ and is monotonically increasing towards $0$ for $\mathcal{E} > \mathcal{E}_c$. Since 
$$
\int_{-1}^{-c^2/2\sqrt{2 \mathcal{E}}} \frac{\sqrt{c^2 - 2 \sqrt{2 \mathcal{E}} x}}{\sqrt{1 - x^2}} d x \sim  
(8 \mathcal{E})^{1/4} \int_{-1}^0 \frac{\sqrt{|x|}dx}{\sqrt{1-x^2}}
\quad \mbox{\rm as} \;\; \mathcal{E}\to \infty
$$
and
$$
\int_{-c^2/2}^{c^2/2} \frac{\sqrt{c^2 - 2 \eta}}{\sqrt{2 \mathcal{E} - \eta^2}} d \eta \sim  
(2 \mathcal{E})^{-1/2} \int_{-c^2/2}^{c^2/2} \sqrt{c^2 - 2 \eta} d \eta
\quad \mbox{\rm as} \;\; \mathcal{E}\to \infty
$$
the first term in the decomposition of $T(\mathcal{E},c)$ is larger than the second term at infinity and we have $T(\mathcal{E},c) = \mathcal{O}(\mathcal{E}^{1/4})$ as $\mathcal{E} \to \infty$ so that 
\begin{align*}
T(\mathcal{E},c) &\to \infty \quad \mbox{\rm as} \;\;\mathcal{E}\to \infty.
\end{align*}
At the same time, the first term in the decomposition of $T(\mathcal{E},c)$ is zero at $\mathcal{E} = \mathcal{E}_c$, hence there exist $\mathcal{E}_*,\mathcal{E}_{**} \in (\mathcal{E}_c,\infty)$ such that (\ref{period-decreasing}) and (\ref{period-increasing}) hold.
Monotonicity (\ref{period-speed})  follows from the positive derivative of (\ref{period-function-sing}) in $c$.
\end{proof}

We are now ready to prove Theorem \ref{theorem-1}.

\begin{proof}[Proof of Theorem \ref{theorem-1}]
	We consider the family of solutions of Lemmas \ref{lemma-2} and \ref{lemma-3} for $\mathcal{E}\in (0,\infty)$ and $\eta < \frac{c^2}{2}$. For every $c > 0$, we select the intersection of the period function $T(\mathcal{E},c)$ of Lemmas \ref{lemma-4}  with the $2\pi$ period on $\mathbb{T}$. It follows from (\ref{period-speed}) that the period function is monotonically increasing in $c$ for every $\mathcal{E} \in (0,\infty)$.
	
	For the smooth periodic solutions with the period function (\ref{period-function}), there exists only one root $\mathcal{E} \in (0,\mathcal{E}_c)$ of  $T(\mathcal{E},c) = 2\pi$ for every $c \in (1,c_*)$ 
	with $c_* = \frac{\pi}{2 \sqrt{2}}$ due to 
	monotonicity (\ref{period-decreasing}) and the limiting values 
	$T(0,c) = 2 \pi c$ and $T(\mathcal{E}_c,c) = 4 \sqrt{2} c$. This gives the first assertion of the theorem with the limit (\ref{small-amplitude}) since the smooth periodic solution shrinks to the center point $(0,0)$ on the $(\eta,\eta')$ plane as $\mathcal{E}\to 0$. 
	At $c = c_*$, we have $\mathcal{E} = \mathcal{E}_{c_*}$ for the root of $T(\mathcal{E},c) = 2\pi$. The unique single-lobe solution (\ref{quadratic}) follows from the unique solution (\ref{parabola}) at $\mathcal{E} = \mathcal{E}_c$ from Lemma \ref{lemma-2} by the translation $u_0 = \pi$ for $u \in [0,\pi]$ and an even reflection on $[-\pi,0]$.
	
	For the singular solutions with the period function (\ref{period-function-sing}), we have a root $\mathcal{E} \in (\mathcal{E}_c,\infty)$ of $T(\mathcal{E},c) = 2\pi$ for every $c \in (c_*,c_{\infty})$ due to monotonicity (\ref{period-decreasing}).
	The value $c_{\infty}$ is obtained from the intersection of 
	$T(\mathcal{E}_*,c)$ with $2 \pi$. The asymptotic expansion (\ref{2-3-singularity}) of the singular single-lobe solutions 
	follows from the expansion (\ref{sol-expansion}) with the translation $u_- = 0$ of the solution in Lemma \ref{lemma-3} for $u \in [0,\pi]$ with $\eta(\pi) = -\sqrt{2 \mathcal{E}}$ and $\eta'(\pi) = 0$ and an even reflection on $[-\pi,0]$.
\end{proof}

Figure \ref{fig-2} illustrates the result of Lemma \ref{lemma-4} and the proof of Theorem \ref{theorem-1}. The period functions (\ref{period-function}) and (\ref{period-function-sing}) can be computed in terms of complete elliptic integrals by using 3.141 (integrals 2 and 9) in \cite{GR}:
$$
T(\mathcal{E},c) = 4 \sqrt{c^2 + 2 \sqrt{2\mathcal{E}}} E\left(\sqrt{\frac{4 \sqrt{2\mathcal{E}}}{c^2 + 2 \sqrt{2\mathcal{E}}}}\right), \quad \mathcal{E} \in (0,\mathcal{E}_c)
$$ 
and
\begin{align*}
T(\mathcal{E},c) &= 4 (2\mathcal{E})^{1/4} \left[ 2 E\left(\sqrt{\frac{c^2 + 2 \sqrt{2\mathcal{E}}}{4 \sqrt{2\mathcal{E}}}} \right) \right. \\
& \qquad \left.
+ \left( \frac{c^2}{2 \sqrt{2 \mathcal{E}}} - 1 \right) K\left(\sqrt{\frac{c^2 + 2 \sqrt{2\mathcal{E}}}{4 \sqrt{2\mathcal{E}}}} \right) \right], \quad \mathcal{E} \in (\mathcal{E}_c,\infty),
\end{align*}
where $K(k)$ and $E(k)$ are complete elliptic integrals of the first and second kind, respectively. The two definitions agree at $T(\mathcal{E}_c,c) = 4 \sqrt{2} c$, where $\mathcal{E}_c = \frac{c^4}{8}$ shown by the black dot in Figure \ref{fig-2}. We can also see from Figure \ref{fig-2} that 
the monotonicity results (\ref{period-decreasing}), (\ref{period-increasing}), and (\ref{period-speed}) hold true with $\mathcal{E}_* = \mathcal{E}_{**}$. The periodic solutions of Theorem \ref{theorem-1} on $\mathbb{T}$ are obtained by the intersection of the plot of the period function with the level $T(\mathcal{E},c) = 2\pi$.

\begin{figure}[htb!]
	\centering
	\includegraphics[width=12cm,height=10cm]{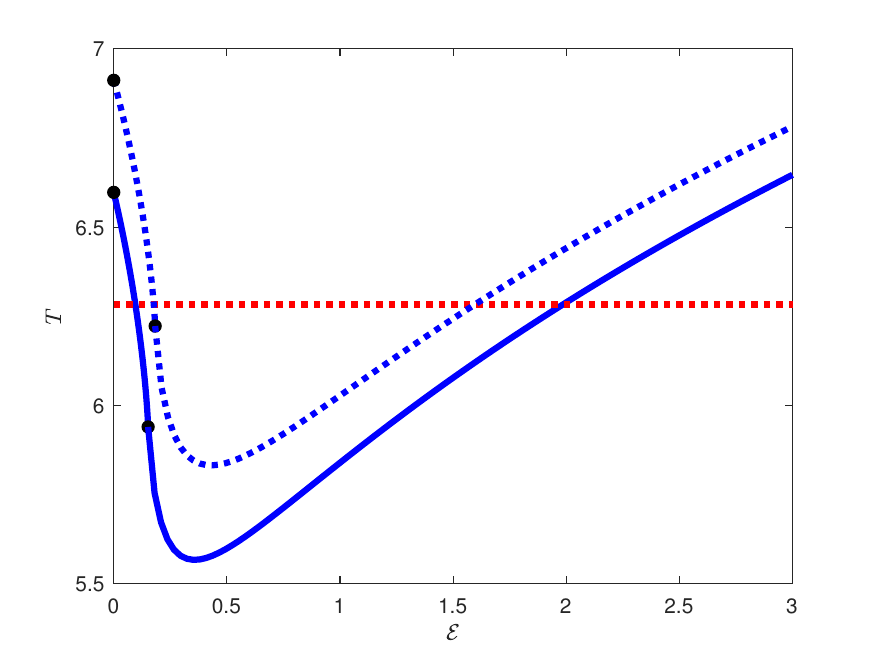}
	\caption{Period function $T$ versus $\mathcal{E}$ for $c = 1.05$ (solid blue) and $c = 1.1$ (dashed blue). Black dots show the values $T(0,c) = 2\pi c$ and 
		$T(\mathcal{E}_c,c) = 4 \sqrt{2}c$. The red dashed line gives the level 
		$T(\mathcal{E},c) = 2\pi$ for periodic solutions on $\mathbb{T}$.}
	\label{fig-2}
\end{figure}

\section{Stability of smooth traveling periodic waves}
\label{sec-3}

Here we consider the unique single-lobe solution with the profile $\eta \in C^{\infty}_{\rm per}(\mathbb{T})$ in Theorem \ref{theorem-1} which exists in the second-order equation (\ref{ode-model}) for $c \in (1,c_*)$ with $c_* = \frac{\pi}{2 \sqrt{2}}$. Theorem \ref{theorem-2} is proven by showing that the periodic solution is a constrained minimizer of the quadratic form 
$\langle \mathcal{L}\hat{\eta}, \hat{\eta} \rangle$ associated with the linear operator 
\begin{equation}
\label{lin-operator}
\mathcal{L} = -\partial_u (c^2 - 2 \eta) \partial_u - 1 + 2 \eta''
\end{equation}
in the function space $H^1_{\rm per}(\mathbb{T}) \cap \mathcal{X}_c$, where 
\begin{equation}
\label{function-space}
\mathcal{X}_c := \left\{ \hat{\eta} \in L^2(\mathbb{T}) : \quad 
\langle 1, \hat{\eta} \rangle = \langle \eta'', \hat{\eta} \rangle = 0 \right\}.
\end{equation}
The minimizer is only degenerate due to the translational symmetry which results 
in the one-dimensional kernel ${\rm Ker}(\mathcal{L}) = {\rm span}(\eta')$ with $\eta' \in \mathcal{X}_c$. 

We start with the count of the negative eigenvalues of $\mathcal{L} : H^2_{\rm per}(\mathbb{T}) \subset L^2(\mathbb{T}) \to L^2(\mathbb{T})$ in the following lemma.

\begin{lemma}
\label{lem-stab-1}
Let $\eta \in C^{\infty}_{\rm per}(\mathbb{T})$ be the profile of the single-lobe solution in Theorem \ref{theorem-1} for $c \in (1,c_*)$ and $\mathcal{L} :  H^2_{\rm per}(\mathbb{T}) \subset L^2(\mathbb{T}) \to L^2(\mathbb{T})$ be the linear operator given by (\ref{lin-operator}). Then, $\mathcal{L}$ has two simple negative eigenvalues and a simple zero eigenvalue, with the rest of its spectrum bounded away from zero. 
\end{lemma}

\begin{proof}
	Since $\mathcal{L} :  H^2_{\rm per}(\mathbb{T}) \subset L^2(\mathbb{T}) \to L^2(\mathbb{T})$ is a self-adjoint Sturm--Liouville operator with $\eta'' \in L^{\infty}(\mathbb{T})$ and $c^2 - 2 \eta(u) > 0$ for every $u \in \mathbb{T}$, its spectrum consists of isolated eigenvalues located on the real line. 
	
For fixed $c \in (1,c_*)$, differentiating (\ref{ode-model}) in $u$ yields 
$\mathcal{L} \eta' = 0$ with $\eta' \in H^2_{\rm per}(\mathbb{T})$. Hence, $\mathcal{L}$ admits a zero eigenvalue with the spatially odd eigenfunction $\eta'$. 
To consider the second linearly independent solution of $\mathcal{L} \hat{\eta} = 0$, we define the family $\{ \eta(u;\mathcal{E})\}_{\mathcal{E}\in (0,\mathcal{E}_c)}$ of spatially even, smooth periodic solutions of the second-order equation (\ref{ode-model}) with 
$$
\eta(0;\mathcal{E}) = \sqrt{2 \mathcal{E}}, \quad 
\partial_u \eta(0;\mathcal{E}) = 0
$$
and 
$$
\eta(T(\mathcal{E},c);\mathcal{E}) = \sqrt{2 \mathcal{E}}, \quad 
\partial_u \eta(T(\mathcal{E},c);\mathcal{E}) = 0,
$$
where $T(\mathcal{E},c)$ is the period function (\ref{period-function}) satisfying (\ref{period-decreasing}). Let $\mathcal{E}(c)$ be the root of the period function $T(\mathcal{E},c) = 2\pi$ in the proof of Theorem \ref{theorem-1} 
such that $\eta(u) = \eta(u;\mathcal{E}(c)) \in C^{\infty}_{\rm per}(\mathbb{T})$. Differentiating (\ref{ode-model}) in $\mathcal{E}$ along the family 
$\{ \eta(u;\mathcal{E})\}_{\mathcal{E}\in (0,\mathcal{E}_c)}$ 
and setting $\mathcal{E}= \mathcal{E}(c)$ yields the second linearly independent solution of $\mathcal{L} \hat{\eta} = 0$. 
The solution is given by the spatially even function $\partial_{\mathcal{E}} \eta(\cdot;\mathcal{E}(c))$ which satisfies
$$
\partial_{\mathcal{E}} \eta(0;\mathcal{E}) = \frac{1}{\sqrt{2 \mathcal{E}}}, \quad 
\partial_u \partial_{\mathcal{E}} \eta(0;\mathcal{E}) = 0
$$
and 
$$
\partial_{\mathcal{E}} \eta(T(\mathcal{E},c);\mathcal{E}) = \frac{1}{\sqrt{2 \mathcal{E}}}, \quad 
\partial_u \partial_{\mathcal{E}} \eta(T(\mathcal{E},c);\mathcal{E}) =  \partial_{\mathcal{E}} T(\mathcal{E},c) \frac{\sqrt{2 \mathcal{E}}}{c^2 - 2 \sqrt{2 \mathcal{E}}},
$$
where we have used (\ref{ode-model}) at $u = T(\mathcal{E},c)$ which implies 
$$
\partial^2_u \eta(T(\mathcal{E},c);\mathcal{E}) = -\frac{\sqrt{2 \mathcal{E}}}{c^2 - 2 \sqrt{2 \mathcal{E}}}.
$$
We thus have $\partial_{\mathcal{E}} \eta(\cdot;\mathcal{E}(c)) \in H^2_{\rm per}(\mathbb{T})$ if and only if $\partial_{\mathcal{E}} T(\mathcal{E}(c),c) = 0$, 
which is impossible due to monotonicity (\ref{period-decreasing}). Hence, $0$ is a simple eigenvalue of $\mathcal{L}$ bounded away from 
the rest of its spectrum in $L^2(\mathbb{T})$. 

To prove that there exist two negative eigenvalues of $\mathcal{L}$ below the zero eigenvalue, we use Proposition 1 in \cite{GMNP} and construct the following  two normalized solutions 
$$
\eta_1(u) := \sqrt{2\mathcal{E}} \partial_{\mathcal{E}} \eta(u;\mathcal{E}) \quad \mbox{\rm and} \quad 
\eta_2(u) := \frac{\eta'(u)}{\eta''(0)},
$$
where $\eta_2(u+2\pi) = \eta_2(u)$ and $\eta_1(u+2\pi) = \eta_1(u) + \theta \eta_2(u)$ with 
$$
\theta := \partial_{\mathcal{E}} T(\mathcal{E}(c),c) \frac{\sqrt{2 \mathcal{E}(c)}}{c^2 - 2 \sqrt{2 \mathcal{E}(c)}} < 0,
$$
due to monotonicity (\ref{period-decreasing}). By Proposition 1 in \cite{GMNP}, 
$0$ is the third eigenvalue of $\mathcal{L}$ with two simple negative eigenvalues below $0$.
\end{proof}

\begin{remark}
	One can prove the assertion of lemma \ref{lem-stab-1} by using small-amplitude expansions. The periodic solution with the profile $\eta \in C^{\infty}_{\rm per}(\mathbb{T})$ is expanded near the trivial solution, see (\ref{small-amplitude}), as 
	$$
	\eta = a \cos(u) - a^2 \sin^2(u) + \mathcal{O}(a^3), \quad c^2 = 1 + \frac{1}{2} a^2 + \mathcal{O}(a^4),
	$$
	where $a > 0$ is a small parameter. Then, $\mathcal{L}$ along the solution family has one negative eigenvalue $-1 + \mathcal{O}(a^2)$ and a small negative eigenvalue $-a^2 + \mathcal{O}(a^4)$ with $0$ being the third eigenvalue. Since ${\rm Ker}(\mathcal{L}) = {\rm span}(\eta')$ along the family of smooth periodic solutions for $c \in (1,c_*)$, the inertia index of $\mathcal{L}$ remains the same for 
	every $c \in (1,c_*)$.
\end{remark}

The next lemma specifies the criterion for the constrained linear operator 
$\mathcal{L}|_{\mathcal{X}_c}$ to be positive, where $\mathcal{L}|_{\mathcal{X}_c} = \mathcal{L} |_{\{ 1, \eta''\}^{\perp}}$ is defined by the two constraints in (\ref{function-space}).  

\begin{lemma}
	\label{lem-stab-2}
	Let $\mathcal{L} :  H^2_{\rm per}(\mathbb{T}) \subset L^2(\mathbb{T}) \to L^2(\mathbb{T})$ be given by (\ref{lin-operator}) as in Lemma \ref{lem-stab-1} and $\mathcal{X}_c$ be the constrained subspace of $L^2(\mathbb{T})$ given by (\ref{function-space}). Then, $\mathcal{L}|_{\mathcal{X}_c}$ has a  simple zero eigenvalue and no negative eigenvalues, with the rest of its spectrum being bounded away from zero, if and only if the mapping 
\begin{equation}
\label{mass-monotonicity}
	(1,c_*) \ni c \mapsto \mathcal{M}(c) := \oint \eta du 
\end{equation}
	is monotonically decreasing for $c \in (1,c_*)$.
\end{lemma}

\begin{proof}
	By Proposition 2 in \cite{GMNP}, we construct the $2$-by-$2$ matrix related to the two constraints in $\mathcal{X}_c$,
	$$
	A := \left[ \begin{matrix} \langle \mathcal{L}^{-1} 1, 1 \rangle & 
	\langle \mathcal{L}^{-1} 1, \eta'' \rangle \\
	\langle \mathcal{L}^{-1} \eta'', 1 \rangle & 
	\langle \mathcal{L}^{-1} \eta'', \eta'' \rangle
	\end{matrix}
	\right].
	$$
	The inverse operator $\mathcal{L}^{-1}$ on ${\rm span}(1,\eta'')$ is well-defined since 
	$$
	{\rm Ker}(\mathcal{L}) = {\rm span}(\eta') \perp {\rm span}(1,\eta'').
	$$
Differentiating (\ref{ode-model}) in $c$ yields 
\begin{equation}
\label{range-1}
\mathcal{L} \partial_c \eta = 2 c \eta'',
\end{equation}
where $\partial_c \eta$ is defined along the family of $2\pi$-periodic solutions $\{ \eta \}_{c \in (1,c_*)}$. The family is smooth in $c$ since 
$T(\mathcal{E},c)$ is $C^1$ on $(0,\mathcal{E}_c) \times (0,c_*)$ 
and $\mathcal{E}(c)$ is $C^1$ on $(0,c_*)$ due to the implicit function theorem for $T(\mathcal{E}(c),c) = 2\pi$ with $\partial_{\mathcal{E}} T(\mathcal{E}(c),c) < 0$. In addition, we have $\mathcal{L} 1 = 2 \eta'' - 1$ so that 
\begin{equation}
\label{range-2}
\mathcal{L} \left( c^{-1} \partial_c \eta - 1 \right) = 1.
\end{equation}
By using (\ref{range-1}) and (\ref{range-2}), we compute 
$$
A = \left[ \begin{matrix} c^{-1} \langle \partial_c \eta, 1 \rangle - 2\pi & 
c^{-1} \langle \partial_c \eta, \eta'' \rangle \\
(2c)^{-1} \langle \partial_c \eta, 1 \rangle & (2c)^{-1}
\langle \partial_c \eta, \eta'' \rangle
\end{matrix}
\right].
$$
Since $A$ is symmetric, we have $\langle \partial_c \eta, 1 \rangle = 2 \langle \partial_c \eta, \eta'' \rangle$. Furthermore, 
$$
\det(A) = - \frac{\pi}{2 c} \langle \partial_c \eta, 1 \rangle = -\frac{\pi}{2c} \mathcal{M}'(c),
$$
where $\mathcal{M}(c)$ is given by (\ref{mass-monotonicity}). 
Since $c > 0$, we have the following trichotomy from Proposition 2 in \cite{GMNP}:
\begin{itemize}
	\item If $\mathcal{M}'(c) > 0$, then $\det(A) < 0$, hence $A$ has one negative and one positive eigenvalue so that $\mathcal{L}|_{\mathcal{X}_c}$ admits one simple negative and a simple zero eigenvalue. \\
	
\item If $\mathcal{M}'(c) = 0$, then $\det(A) = 0$ but ${\rm tr}(A) < 0$, hence $A$ has one negative and one zero eigenvalue so that $\mathcal{L}|_{\mathcal{X}_c}$ admits a double zero eigenvalue and no negative eigenvalues. \\

\item If $\mathcal{M}'(c) < 0$, then $\det(A) > 0$, hence $A$ has two negative eigenvalues so that $\mathcal{L}|_{\mathcal{X}_c}$ admits  a simple zero eigenvalue and no negative eigenvalues.
\end{itemize}
The last case yields the assertion of the lemma.
\end{proof}

\begin{remark}
	Due to constraint (\ref{zero-mean-toy}), we have 
\begin{equation}
\label{M-Q}
	M(\eta) = \oint \eta du = - \oint (\partial_u \eta)^2 du = - Q(\eta), 
\end{equation}
	so that the criterion in Lemma \ref{lem-stab-2} is equivalent to the mapping of 
	$$
	(1,c_*) \ni c \mapsto \mathcal{Q}(c) := \oint (\eta')^2 du 
	$$
	being monotonically increasing for $c \in (1,c_*)$.
\end{remark}

Since $\mathcal{M}(c) \to 0$ as $c \to 1$ by (\ref{small-amplitude}) 
and $\mathcal{M}(c) < 0$ by (\ref{M-Q}), it is clear that 
$\mathcal{M}'(c) < 0$ for $c \gtrsim 1$. The next lemma asserts that 
$\mathcal{M}'(c) < 0$ for every $c \in (1,c_*)$.

\begin{lemma}
	\label{lem-stab-3}
The mapping (\ref{mass-monotonicity}) is monotonically decreasing for every $c \in (1,c_*)$.
\end{lemma}

\begin{proof}
Differentiating $T(\mathcal{E}(c),c) = 2\pi$ in $c$ yields
	$$
	\partial_c T(\mathcal{E}(c),c) + \mathcal{E}'(c) \partial_{\mathcal{E}} T(\mathcal{E}(c),c) = 0,
	$$
where
\begin{align*}
	\partial_c T(\mathcal{E},c) &= 2c \int_{-\sqrt{2 \mathcal{E}}}^{\sqrt{2 \mathcal{E}}} \frac{d \eta}{\sqrt{2 \mathcal{E} - \eta^2} \sqrt{c^2 - 2 \eta}}, \\
	\partial_{\mathcal{E}} T(\mathcal{E},c) &= -\mathcal{E}^{-1} \int_{-\sqrt{2 \mathcal{E}}}^{\sqrt{2 \mathcal{E}}} \frac{\eta d \eta}{\sqrt{2 \mathcal{E} - \eta^2} \sqrt{c^2 - 2 \eta}},
\end{align*}
and we have used (\ref{tech-eq}) with the substitution $\eta = \sqrt{2 \mathcal{E}}x$. By using the same substitution, we define $\mathcal{M}(c) \equiv \mathcal{M}(\mathcal{E}(c),c)$ with 
$$
\mathcal{M}(\mathcal{E},c) := 2 \int_{-\sqrt{2 \mathcal{E}}}^{\sqrt{2 \mathcal{E}}} \frac{\eta \sqrt{c^2 - 2 \eta} d \eta}{\sqrt{2 \mathcal{E} - \eta^2}} = 
2 \sqrt{2 \mathcal{E}} \int_{-1}^{1} \frac{x \sqrt{c^2 - 2 \sqrt{2\mathcal{E}} x} d x}{\sqrt{1-x^2}},
$$
from which we obtain 
\begin{align*}
\partial_c \mathcal{M}(\mathcal{E},c) &= 2c \int_{-\sqrt{2 \mathcal{E}}}^{\sqrt{2 \mathcal{E}}} \frac{\eta d \eta}{\sqrt{2 \mathcal{E} - \eta^2} \sqrt{c^2 - 2 \eta}}, \\
\partial_{\mathcal{E}} \mathcal{M}(\mathcal{E},c) &= \mathcal{E}^{-1} \int_{-\sqrt{2 \mathcal{E}}}^{\sqrt{2 \mathcal{E}}} \frac{\eta \sqrt{c^2 - 2 \eta} d \eta}{\sqrt{2 \mathcal{E} - \eta^2} } - \mathcal{E}^{-1} \int_{-\sqrt{2 \mathcal{E}}}^{\sqrt{2 \mathcal{E}}} \frac{\eta^2 d \eta}{\sqrt{2 \mathcal{E} - \eta^2} \sqrt{c^2 - 2 \eta}}.
\end{align*}	
Since $\partial_{\mathcal{E}} T(\mathcal{E}(c),c) < 0$, we obtain 
\begin{align*}
\mathcal{M}'(c) &= \partial_c \mathcal{M}(\mathcal{E}(c),c) + \mathcal{E}'(c) \partial_{\mathcal{E}} \mathcal{M}(\mathcal{E}(c),c) \\
&= \frac{2c}{\mathcal{E}(c) |\partial_{\mathcal{E}} T(\mathcal{E}(c),c)|} \Delta(\mathcal{E}(c),c),
\end{align*}
where 
\begin{align*}
\Delta(\mathcal{E},c) &:= \left( \int_{-\sqrt{2 \mathcal{E}}}^{\sqrt{2 \mathcal{E}}} \frac{\eta d \eta}{\sqrt{2 \mathcal{E} - \eta^2} \sqrt{c^2 - 2 \eta}} \right)^2 \\
& \quad + \left( \int_{-\sqrt{2 \mathcal{E}}}^{\sqrt{2 \mathcal{E}}} \frac{d \eta}{\sqrt{2 \mathcal{E} - \eta^2} \sqrt{c^2 - 2 \eta}} \right) 
\left( \int_{-\sqrt{2 \mathcal{E}}}^{\sqrt{2 \mathcal{E}}} \frac{\eta (c^2 - 3 \eta) d \eta}{\sqrt{2 \mathcal{E} - \eta^2} \sqrt{c^2 - 2 \eta}}
\right).
\end{align*}
We show that $\Delta(\mathcal{E},c) < 0$, which implies that $\mathcal{M}'(c) < 0$. Indeed, since 
$$
\frac{\eta \sqrt{c^2 - 2 \eta}}{\sqrt{2\mathcal{E} - \eta^2}} < \frac{|\eta| \sqrt{c^2 + 2 |\eta|}}{\sqrt{2\mathcal{E} - \eta^2}}, \quad \eta \in (0,\sqrt{2\mathcal{E}}),
$$
we have 
$$
\left( \int_{-\sqrt{2 \mathcal{E}}}^{\sqrt{2 \mathcal{E}}} \frac{d \eta}{\sqrt{2 \mathcal{E} - \eta^2} \sqrt{c^2 - 2 \eta}} \right) 
\left( \int_{-\sqrt{2 \mathcal{E}}}^{\sqrt{2 \mathcal{E}}} \frac{\eta \sqrt{c^2 - 2 \eta} d \eta}{\sqrt{2 \mathcal{E} - \eta^2}} \right) < 0.
$$
The remaining part of $\Delta(\mathcal{E},c)$ is also negative since 
\begin{align*}
& \quad \left( \int \frac{\eta d \eta}{\sqrt{2 \mathcal{E} - \eta^2} \sqrt{c^2 - 2 \eta}} \right)^2 - \left( \int \frac{d \eta}{\sqrt{2 \mathcal{E} - \eta^2} \sqrt{c^2 - 2 \eta}} \right) 
\left( \int\frac{\eta^2 d \eta}{\sqrt{2 \mathcal{E} - \eta^2} \sqrt{c^2 - 2 \eta}}
\right) \\
& = \int\int \frac{\eta_1 \eta_2 - \eta_2^2}{\sqrt{2 \mathcal{E} - \eta_1^2} \sqrt{c^2 - 2 \eta_1} \sqrt{2 \mathcal{E} - \eta_2^2} \sqrt{c^2 - 2 \eta_2}} d \eta_1 d \eta_2 \\
& = -\frac{1}{2} \int\int \frac{(\eta_1-\eta_2)^2}{\sqrt{2 \mathcal{E} - \eta_1^2} \sqrt{c^2 - 2 \eta_1} \sqrt{2 \mathcal{E} - \eta_2^2} \sqrt{c^2 - 2 \eta_2}} d \eta_1 d \eta_2 < 0,
\end{align*}
where the integrations are defined on $[-\sqrt{2\mathcal{E}},\sqrt{2\mathcal{E}}]$. 
Hence $\Delta(\mathcal{E}(c),c) < 0$ and the assertion of the lemma has been proven.
\end{proof}

We are now ready to prove Theorem \ref{theorem-2}.

\begin{proof}[Proof of Theorem \ref{theorem-2}]
First, we prove that the two constraints (\ref{constraints-linear}) are preserved in the time evolution of the linearized equation (\ref{lin-model}).
Since $\Pi_0 \partial_u^{-1} \Pi_0$ is defined on zero-mean functions with the zero-mean constraint, taking the mean value of (\ref{lin-model}) yields 
$$
2c \frac{d}{dt} \langle 1, \hat{\eta} \rangle = 0.
$$
Multiplying (\ref{lin-model}) by $\eta''$ and integrating by parts, we obtain 
for any solution $\hat{\eta} \in C^0(\mathbb{R},H^1_{\rm per}(\mathbb{T}))$ 
\begin{align*}
2c \frac{d}{dt}  \langle \eta'', \hat{\eta} \rangle &= 
\langle (c^2 - 2 \eta) \eta'', \partial_u \hat{\eta} \rangle - 
\langle (1-2\eta'') \eta', \hat{\eta} \rangle \\
&= \langle (\eta')^2 - \eta, \partial_u \hat{\eta} \rangle - 
\langle (1-2\eta'') \eta', \hat{\eta} \rangle \\
&= \langle \eta' - 2 \eta' \eta'', \hat{\eta} \rangle - 
\langle (1-2\eta'') \eta', \hat{\eta} \rangle \\
&= 0,
\end{align*}
where (\ref{ode-model}) has been used with $\eta \in C^{\infty}_{\rm per}(\mathbb{T})$. Hence the two constraints (\ref{constraints-linear}) are preserved in time and the solution  $\hat{\eta} \in C^0(\mathbb{R},H^1_{\rm per}(\mathbb{T}))$ of the linearized equation (\ref{lin-model}) with $\hat{\eta}(\cdot,0) = \hat{\eta}_0$ and $\hat{\eta}_0 \in \mathcal{X}_c$ satisfies $\hat{\eta}(\cdot,t) \in \mathcal{X}_c$. Thus, we have 
$$
\langle 1, \hat{\eta}(\cdot,t) \rangle = 0, \quad 
\langle \eta'', \hat{\eta}(\cdot,t) \rangle = 0, \quad t \in \mathbb{R}.
$$
If we further decompose 
$$
\hat{\eta}(\cdot,t) = a(t) \eta' + w(\cdot,t), \quad t \in \mathbb{R},
$$
then $w(\cdot,t) \in H^1_{\rm per}(\mathbb{T}) \cap \mathcal{X}_c$ for $t \in \mathbb{R}$ satisfies the additional constraint 
$$
\langle \eta', w(\cdot,t) \rangle = 0, \quad t \in \mathbb{R}.
$$

Next, existence and uniqueness of solutions $\hat{\eta} \in C^0(\mathbb{R},H^1_{\rm per}(\mathbb{T}))$ of the linearized equation (\ref{lin-model}) such that $\hat{\eta}(\cdot,0) = \hat{\eta}_0$ follows by the energy method \cite{RR}.
The energy quadratic form $\langle \mathcal{L} \hat{\eta},\hat{\eta} \rangle$ is bounded and conserved for the solution $\hat{\eta} \in C^0(\mathbb{R},H^1_{\rm per}(\mathbb{T}))$ of the linearized equation (\ref{lin-model}).  Since $\mathcal{L} \eta' = 0$, 
we get 
\begin{align*}
\langle \mathcal{L} w(\cdot,t), w(\cdot,t) \rangle = 
\langle \mathcal{L} \hat{\eta}(\cdot,t), \hat{\eta}(\cdot,t) \rangle = \langle \mathcal{L}\hat{\eta}_0, \hat{\eta}_0 \rangle \leq \beta \| \hat{\eta}_0 \|^2_{H^1_{\rm per}},
\end{align*} 
for some fixed  $\beta > 0$. By Lemmas \ref{lem-stab-2} and \ref{lem-stab-3}, $\langle \mathcal{L} w(\cdot,t), w(\cdot,t) \rangle$ is coercive for $w(\cdot,t) \in H^1_{\rm per}(\mathbb{T}) \cap \mathcal{X}_c$ 
and is non-degenerate if $w(\cdot,t)$ is orthogonal to $\eta'$. Hence, we get the lower bound with some fixed $\alpha > 0$:
\begin{align*}
\alpha \| w(\cdot,t) \|^2_{H^1_{\rm per}} \leq \langle \mathcal{L} w(\cdot,t), w(\cdot,t) \rangle \leq \beta \| \hat{\eta}_0 \|^2_{H^1_{\rm per}},
\end{align*}
which implies the first estimate in (\ref{lin-stability}). In addition, we get from (\ref{lin-model}) due to $\mathcal{L} \eta' = 0$,
$$
2c a'(t) \eta' + 2c \partial_t w = - \Pi_0 \partial_u^{-1} \Pi_0 \mathcal{L} w,
$$
which allows us to control the unique $a \in C^1(\mathbb{R},\mathbb{R})$ from the bound 
\begin{align*}
2c a'(t) \| \eta' \|^2_{L^2} &= 
\langle (c^2 - 2 \eta) \eta', \partial_u w(\cdot,t) \rangle -
\langle (1-2\eta'') \eta, w(\cdot,t) \rangle \\
&\leq \gamma \| w(\cdot,t) \|_{H^1_{\rm per}},
\end{align*}
for some fixed $\gamma > 0$, which yields the second estimate in (\ref{lin-stability}).
\end{proof}

\begin{figure}[htb!]
	\centering
	\includegraphics[width=7cm,height=6cm]{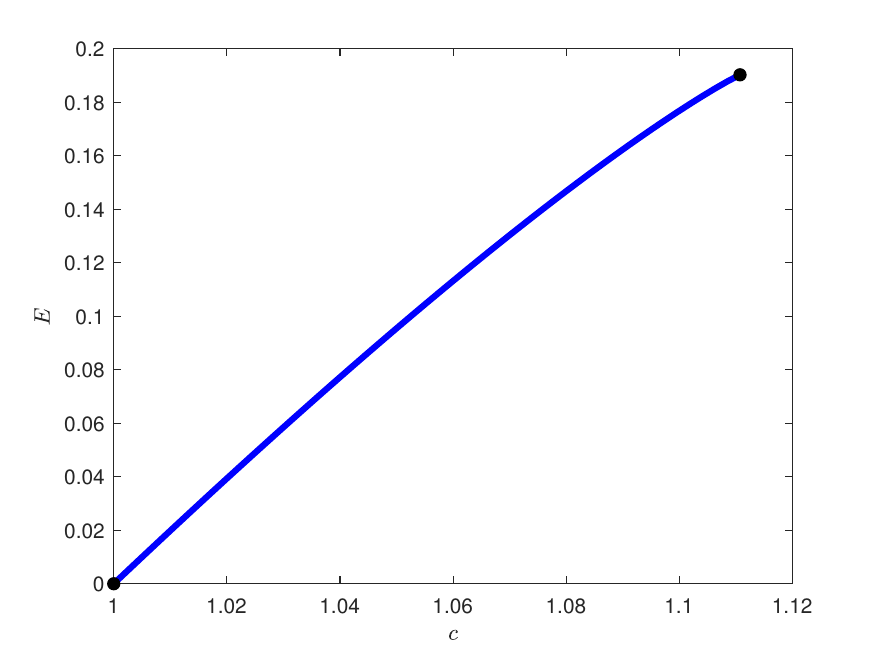} \hspace{1cm}
	\includegraphics[width=7cm,height=6cm]{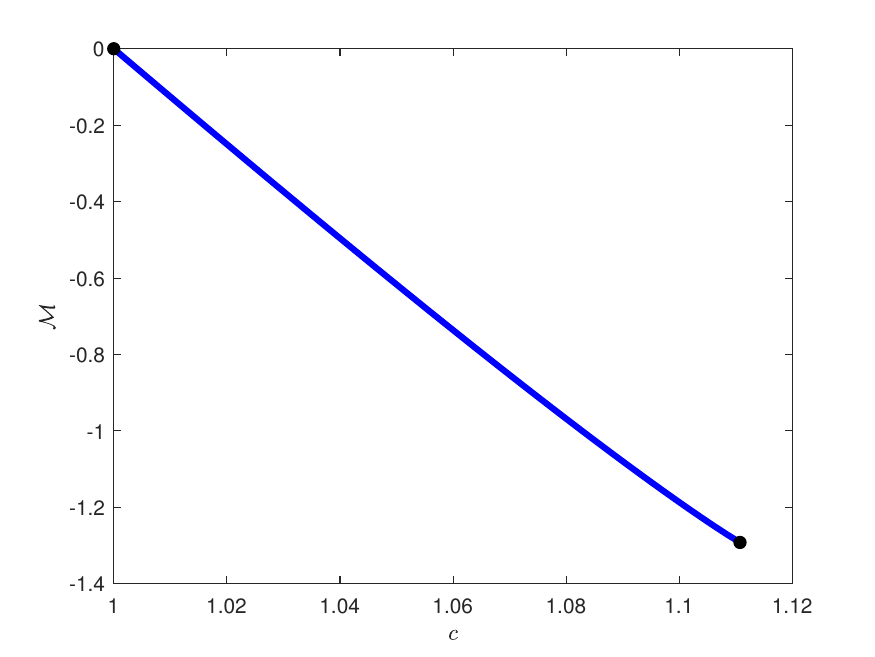}
	\caption{Dependence of $\mathcal{E}$ (left) and $\mathcal{M}$ (right) versus $c$ along the family of solutions of $T(\mathcal{E}(c),c) = 2\pi$. The black dots show the values $\mathcal{E}(1) = 0$, $\mathcal{M}(1) = 0$ and $\mathcal{E}(c_*) = \frac{\pi^4}{512}$,  $\mathcal{M}(c_*) = -\frac{\pi^3}{24}$.}
	\label{fig-3}
\end{figure}

Figure \ref{fig-3} display the dependence of $\mathcal{E}$ and $\mathcal{M}$ versus $c$ for $c \in (1,c_*)$ computed along the family of solutions of $T(\mathcal{E}(c),c) = 2\pi$. The mass integral can be computed in terms of complete elliptic integrals by using 3.141 (integral 20) in \cite{GR}:
\begin{align*}
\mathcal{M}(\mathcal{E},c) &= -2 \int_{-\sqrt{2\mathcal{E}}}^{\sqrt{2\mathcal{E}}} \frac{\sqrt{2\mathcal{E} - \eta^2}}{\sqrt{c^2 - 2 \eta}} d \eta \\
&= -4 \mathcal{E} \int_{-1}^1 \frac{\sqrt{1-x^2}}{\sqrt{c^2 - 2 \sqrt{2\mathcal{E}} x}} dx \\
&= - \frac{2}{3} \sqrt{c^2 + 2 \sqrt{2\mathcal{E}}} 
\left[ c^2 E\left(\sqrt{\frac{4 \sqrt{2\mathcal{E}}}{c^2 + 2 \sqrt{2\mathcal{E}}}}\right) - 
(c^2 - 2 \sqrt{2 \mathcal{E}}) K\left(\sqrt{\frac{4 \sqrt{2\mathcal{E}}}{c^2 + 2 \sqrt{2\mathcal{E}}}}\right) \right], 
\end{align*}
where $K(k)$ and $E(k)$ are complete elliptic integrals of the first and second kind, respectively. The values of $\mathcal{E} = \mathcal{E}(c)$ are computed numerically from $T(\mathcal{E}_c,c) = 2\pi$ by a root-finding algorithm. Figure \ref{fig-3} illustrates the monotonicity result of Lemma \ref{lem-stab-3}. Since $\mathcal{E}(c) \to 0$ as $c \to 1$ follows from (\ref{small-amplitude}), we have $\mathcal{M}(c) \to 0$ as $c \to 1$. On the other hand, 
$\mathcal{E}(c_*) \to \frac{c_*^4}{8} = \frac{\pi^4}{512}$ as $c \to c_*$ follows by Lemma \ref{lemma-4} and we compute from (\ref{quadratic}) that 
\begin{align*}
\mathcal{M}(c_*) &= \frac{1}{8} \int_0^{\pi} (\pi^2 - 4 \pi u + 2u^2) du = -\frac{\pi^3}{24},
\end{align*}
which agrees well with the numerical data in Figure \ref{fig-3}. 

\vspace{0.25cm}

{\bf Numerical methods.} The numerical data in Figure \ref{fig-0} is an extended version of the left panel of Figure \ref{fig-3}, where all roots of $T(\mathcal{E},c) = 2\pi$ 
have been computed numerically from a bisection method, see Figure \ref{fig-2}.
The numerical data in Figure \ref{fig-profiles} was obtained from finding roots of the implicit function 
$$
|u| = \int_{\eta}^{\eta_{\rm max}} \frac{\sqrt{c^2 - 2 \eta}}{\sqrt{2\mathcal{E}(c)-\eta^2}} d \eta,
$$
where $\eta_{\rm max} := \sqrt{2 \mathcal{E}(c)}$ for smooth profiles (left panel) and $\eta_{\rm max} := \frac{c^2}{2}$ for singular profiles (right panel) and $\mathcal{E}(c)$ is a root of $T(\mathcal{E}(c),c) = 2\pi$ obtained on the lower part of the bifurcation diagram in Figure \ref{fig-0}.

\vspace{0.25cm}

{\bf Acknowledgement.} This work is a part of the undergraduate thesis of S. Locke at McMaster University (2023-2024). The authors thank S. Dyachenko, P. Lushnikov, J. Weber, and X. Zhao for many discussions related to the content of this work. Figure \ref{fig-profiles} was prepared by S. Wang as a part of an undergraduate summer project (2024). D.~E.~Pelinovsky acknowledges the funding of this study provided by the grant No.~FSWE-2023-0004 through the State task program in the sphere of scientific activity of the Ministry of Science and Higher Education of the Russian Federation and grant No.~NSH-70.2022.1.5 for the State support of leading Scientific Schools of the Russian Federation.

\vspace{0.25cm}

{\bf Declaration of interests.} The authors report no conflict of interest.

\vspace{0.25cm}

{\bf Data availability statement.} The data that support the findings of this study are available upon request from the authors.

\appendix

\section{Euler equations after a conformal transformation}
\label{appendix}

Let $y = \eta(x,t)$ be the profile for the free surface of an incompressible and irrotational fluid in the $2\pi$-periodic domain and assume a flat bottom at $y = -h_0$, where the vertical velocity vanishes. For a proper definition of the fluid depth $h_0$, we add the zero-mean constraint on the free surface, that is, 
\begin{equation}
\label{zero-mean}
\oint \eta(x,t) dx = 0,
\end{equation}
which is invariant in the time evolution of Euler's equations. 

Let $\varphi(x,y,t)$ be the velocity potential, which satisfies the Laplace equation in the time-dependent spatial domain 
$$
D_{\eta} := \left\{ (x,y) \in \R^2 : \quad -\pi \leq x \leq \pi, \quad -h_0 \leq y \leq \eta(x,t) \right\}
$$
subject to the periodic boundary conditions at $x = \pm \pi$ and the Neumann boundary condition at $y = -h_0$. The formulation of the water wave problem is completed by two additional (kinematic and dynamic) conditions 
at the free surface $y = \eta(x,t)$:
\begin{align}
\label{euler}
\left.  \begin{array}{r} \displaystyle \eta_t + \varphi_x \eta_x - \varphi_y = 0, \\ \displaystyle
\varphi_t + \frac{1}{2} (\varphi_x)^2 + \frac{1}{2} (\varphi_y)^2 + \eta = 0, \end{array} \right\} \qquad \mbox{\rm at} \;\; y = \eta(x,t), 
\end{align}  
where the gravitational constant $g$ is set to unity for convenience. 

The method of conformal transformations is used to map the spatial domain $D_{\eta}$ to the flat domain 
$$
\mathcal{D} := \left\{ (u,v) \in \R^2 : \quad -\pi \leq u \leq \pi, \quad -h \leq v \leq 0 \right\},
$$
where $h$ may be different from $h_0$. The transformation is based on the conformal mapping $x + i y = z(u+iv,t)$, where $w := u + iv$ is a new complex variable and $z \in C^{\omega}(\mathcal{D})$ is a holomorphic function, the real and imaginary parts of which satisfy the Cauchy--Riemann equations 
$$
\frac{\partial x}{\partial u} = \frac{\partial y}{\partial v}, \qquad 
\frac{\partial x}{\partial v} = -\frac{\partial y}{\partial u}.
$$
To preserve the flat bottom $y = -h_0$ at $v = -h$, one needs to add 
the Neumann condition $\partial_v x |_{v = -h} = 0$, which ensures that 
$y(u,-h,t) = -h_0$ is $u$-independent. In addition, we require 
$x - u$ and $y - v$ be $2\pi$-periodic functions of $u \in \mathbb{T} := \mathbb{R} \backslash (2 \pi \mathbb{Z})$ to ensure that  $x(\pi,v,t) - x(-\pi,v,t) = 2\pi$.

Abusing notations we refer to $x = x(u,t)$ and $y = \eta(u,t)$ at the top boundary of $\mathcal{D}$:
$$
x(u,t) + i \eta(u,t) = z(u,t), \qquad \mbox{\rm at } \;\; v = 0.
$$
Similarly, we abuse notations for the velocity potential $\varphi(u,v,t)$ and define 
$$
\xi(u,t) := \varphi(u,v=0,t)
$$ 
on the flat top boundary of $\mathcal{D}$. 
Since the conformal transformation preserves the periodic boundary conditions and the zero vertical velocity condition at $v = -h$, the Laplace equation 
can be solved with the following Fourier series 
$$
\varphi(u,v,t) = \sum_{n \in \mathbb{Z}} \hat{\xi}_n(t) e^{i n u} \; \frac{\cosh (n (v + h))}{\cosh(n h)},
$$
where $\hat{\xi}_n(t)$ is the Fourier coefficient for $\xi(u,t) = \varphi(u,v=0,t)$.  Similarly, we obtain 
\begin{align*}
x(u,v,t) &= u + \sum_{n \in \mathbb{Z}} \hat{x}_n(t) e^{i n u} \; \frac{\cosh (n(v + h))}{\cosh(n h)}, \\
y(u,v,t) &= v + \hat{\eta}_0 + \sum_{n \in \mathbb{Z}} \hat{x}_n(t) e^{i n u} \; i \; \frac{\sinh (n(v + h))}{\cosh(n h)}.
\end{align*}
It follows from $y(u,-h,t) = -h_0$ that $\hat{\eta}_0 = h - h_0$. If $\hat{\eta}_0(t) = \frac{1}{2\pi} \oint \eta(u,t) du$ depends on time $t$, so is $h(t)$ which satisfies $\partial_t \eta (u,-h,t) - h'(t) \partial_v \eta(u,-h,t)$ for all $u \in \mathbb{T}$.

Reducing the Fourier series for $x(u,v,t)$ and $y(u,v,t)$ on $v = 0$ 
yields
\begin{align*}
x(u,t) &= u + \hat{x}_0(t) + \sum_{n \in \mathbb{Z} \backslash \{0\}} \hat{\eta}_n(t) e^{i n u} \; (-i) \coth(nh), \\
\eta(u,t) &= \hat{\eta}_0(t) + \sum_{n \in \mathbb{Z} \backslash \{0\}} \hat{\eta}_n(t) e^{i n u},
\end{align*}
with the correspondence $\hat{\eta}_n(t) = i \tanh(nh) \hat{x}_n(t)$ for $n \in \mathbb{Z} \backslash \{0\}$. 

Let us introduce the nonlocal operator $T_h$ with the Fourier symbol given by 
$$
\widehat{\left( T_h \right)}_n = i \tanh(hn), \qquad n \in \mathbb{Z},
$$
so that $\eta = \hat{\eta}_0 + T_h (x - u)$. The inverse of $T_h$ is only defined on the zero-mean functions with the Fourier symbol given by 
$$
\widehat{\left( T_h^{-1} \right)}_n = \left\{ \begin{array}{cl} -i \coth(hn), &\quad n \in \mathbb{Z} \backslash \{0\},  \\ 0, & \quad n = 0. \end{array} \right.
$$
Inverting $\eta = \hat{\eta}_0 + T_h (x - u)$ yields $x = u + \hat{x}_0 + T_h^{-1} \eta$ and 
$$
x_u = 1 + K_h \eta, 
$$
where $K_h := T_h^{-1} \partial_u$ is a linear, self-adjoint, positive operator on $L^2(\mathbb{T})$. We set $\hat{x}_0 = 0$ in $x = u + \hat{x}_0 + T_h^{-1} \eta$ without loss of generality. 

The equations of motion can be derived from the following Lagrangian, see \cite{DZ1} and \cite[Appendix A]{Lush1} for $h = \infty$, 
$$
\mathcal{L}(\xi,\eta,x) := \oint \xi(\eta_t x_u - \eta_u x_t) du + 
\frac{1}{2} \oint \xi T_h \xi_u du - \frac{1}{2} \oint \eta^2 x_u du + 
\oint f (\eta - T_h(x-u)) du, 
$$
where $f$ is the Lagrange multiplier with the zero mean value $\oint f du = 0$. Variation of $\mathcal{L}$ in $\xi$, $\eta$, and $x$
yields the system of equations 
\begin{equation}
\label{eq-motion}
\left\{ \begin{array}{l}
\eta_t x_u - \eta_u x_t + T_h \xi_u = 0, \\
-\xi_t x_u + \xi_u x_t - \eta x_u + f = 0, \\
\xi_t \eta_u - \xi_u \eta_t + \eta \eta_u + T_h f = 0,
\end{array}
\right. 
\end{equation}
with the additional constraint due to the reduction $h = h_0 + \frac{1}{2\pi} \oint \eta du$:
\begin{equation}
\label{add-constraints}
\frac{1}{2} \oint \xi (\partial_h T_h) \xi_u du -
\oint f (\partial_h T_h) (x-u) du = 0.
\end{equation}

Taking mean values in each equation of system (\ref{eq-motion}) and integrating by parts yields three conserved quantities:
\begin{align}
\label{conserved-mass}
M_1(\eta) &= \oint \eta x_u du = \oint \eta (1+K_h \eta) du = 0, \\
\label{conserved-mass-2}
M_2(\xi,\eta) &= \oint \xi x_u du = \oint \xi (1 + K_h \eta) du, \\
M_3(\xi,\eta) &= \oint \xi \eta_u du,
\label{conserved-mass-3}
\end{align}
where the constraint $M_1(\eta) = 0$ follows from the zero-mean constraint (\ref{zero-mean}) in physical coordinates. We express $f$ from the second equation of system (\ref{eq-motion}), substitute it into the third equation, and invert $T_h$ on the periodic functions with zero mean. This transforms the system (\ref{eq-motion}) to the following system of two equations for $\xi$ and $\eta$:
\begin{equation}
\label{time-Bab-eq}
\left\{ \begin{array}{l}
\eta_t x_u - \eta_u x_t + T_h \xi_u = 0, \\
\xi_t x_u - \xi_u x_t + \eta (1 + K_h \eta)  + T_h^{-1} ( \xi_t \eta_u - \xi_u \eta_t + \eta \eta_u ) = 0.
\end{array}
\right. 
\end{equation}
The constraint (\ref{add-constraints}) is rewritten in the equivalent form:
\begin{equation}
\label{add-constraints-last}
\frac{1}{2} \oint \xi (\partial_h T_h) \xi_u du +
\oint T_h^{-1}(\xi_t \eta_u - \xi_u \eta_t + \eta \eta_u) (\partial_h T_h) (T_h^{-1} \eta) du = 0.
\end{equation}

The conserved energy of system (\ref{time-Bab-eq}) is given by 
\begin{equation}
\label{conserved-Ham}
H(\xi,\eta) = \oint \left[ \xi T_h \xi_u - \eta^2 (1 + K_h \eta) \right] du. 
\end{equation}
To derive (\ref{conserved-Ham}), we multiply the first equation of system (\ref{time-Bab-eq}) by $\xi_t$ and the second equation by $\eta_t$, integrate over the period, and subtract one equation from another. After integration by parts, we get $\frac{d}{dt} H(\xi,\eta) = 0$ if and only if 
\begin{equation*}
-\frac{1}{2} h'(t) \oint \xi (\partial_h T_h) \xi_u du -
h'(t) \oint \eta \eta_u (\partial_h T_h^{-1}) \eta du 
- \oint (\xi_t \eta_u - \xi_u \eta_t ) (x_t - T_h^{-1} \eta_t) du = 0.
\end{equation*}
Since $x_t - T_h^{-1} \eta_t = h'(t) (\partial_h T_h^{-1}) \eta$,  $
\partial_h T_h^{-1} = - T_h^{-1} (\partial_h T_h) T_h^{-1}$, and $T_h^{-1}$ is skew-adjoint, the last constraint is identical to the constraint (\ref{add-constraints-last}) for every $h'(t)$. This proves the conservation of $H(\xi,\eta)$. 

The conserved quantities (\ref{conserved-mass}), (\ref{conserved-mass-2}), (\ref{conserved-mass-3}), and (\ref{conserved-Ham}) coincide with the conserved quantities for Euler's equation in physical coordinates \cite{BO82}, see also \cite{DZ1} for $h = \infty$.

In order to introduce the scalar model (\ref{single-eq}), we rewrite (\ref{time-Bab-eq}) in the reference frame moving with the wave speed $c$:
\begin{equation}
\label{system-appendix}
\left\{ \begin{array}{l}
\eta_t x_u - \eta_u x_t + T_h \xi_u - c \eta_u = 0, \\
\xi_t x_u - \xi_u x_t + \eta (1 + K_h \eta) - c \xi_u  + T_h^{-1} ( \xi_t \eta_u - \xi_u \eta_t + \eta \eta_u ) = 0,
\end{array}
\right. 
\end{equation}
where $u$ now stands for $u - ct$ and we have used the chain rule with 
\begin{align*}
\xi_t & \to \xi_t - c \xi_u, \\
\eta_t & \to \eta_t - c \eta_u, \\
x_t & \to x_t + c - c x_u.
\end{align*}
We introduce a change of variables by 
\begin{equation}
\label{decomposition}
\xi = c T_h^{-1} \eta + \zeta,
\end{equation}
after which the system (\ref{system-appendix}) can be rewritten in the form:
\begin{equation}
\left\{ \begin{array}{l}
\eta_t x_u - \eta_u x_t + T_h \zeta_u = 0, \\
\zeta_t x_u - \zeta_u x_t + \eta (1 + K_h \eta) - c \zeta_u + c x_t - c^2 K_h \eta \\
\qquad \qquad + T_h^{-1} ( \zeta_t \eta_u - \zeta_u \eta_t + \eta \eta_u + c \eta_u x_t - c \eta_t K_h \eta ) = 0.
\end{array}
\right. 
\label{system-eta-zeta}
\end{equation}
Substituting $- c \zeta_u = c T_h^{-1} (\eta_t x_u - \eta_u x_t)$ to the second equation of system (\ref{system-eta-zeta}) 
and taking the derivative of $x = u + T_h^{-1} \eta$ in $t$ yields 
\begin{align}
\notag
& \zeta_t x_u - \zeta_u x_t + T_h^{-1} (\zeta_t \eta_u - \zeta_u \eta_t ) 
+ h'(t) (\partial_h T_h^{-1}) \eta \\
& \qquad 
+ 2 c T_h^{-1} \eta_t  - c^2 K_h \eta  + \eta (1 + K_h \eta)  + \frac{1}{2} K_h \eta^2  = 0.
\label{full-single-eq}
\end{align}
The scalar model (\ref{single-eq}) follows by ignoring the constraint (\ref{add-constraints-last}) and the first equation of system (\ref{system-eta-zeta}) and by setting $\zeta \equiv 0$ and $h'(t) \equiv 0$ in (\ref{full-single-eq}). Babenko's equation (\ref{trav-Bab-eq}), which is the exact equation for traveling waves \cite{Babenko}, corresponds to the time-indepedent solutions of (\ref{system-eta-zeta}) and (\ref{full-single-eq}) with $\zeta \equiv 0$ and $h'(t) \equiv 0$ since $u$ in (\ref{single-eq}) stands for $u - ct$.

\section{Introducing the local model}
\label{app-B}

One popular model for fluid motion is the intermediate long-wave (ILW) equation written in the form 
\begin{equation}
\label{ILW}
\partial_t \eta + \eta \partial_u \eta = \mathcal{K}_h(\partial_u \eta)
\end{equation} 
where $\mathcal{K}_h$ is defined by the Fourier symbol 
\begin{equation*}
\widehat{\left( \mathcal{K}_h \right)}_n = 
\left\{ \begin{array}{cl} n \coth(hn), &\quad n \in \mathbb{Z} \backslash \{0\},  \\ h^{-1}, & \quad n = 0. \end{array} \right.
\end{equation*}
In comparison with (\ref{operator-K}), we have the correspondence 
$$
\mathcal{K}_h = K_h + \frac{1}{2\pi h} \oint \; \cdot \; du.
$$
The ILW equation (\ref{ILW}) is integrable by inverse scattering and many results on well-posedness and dynamics of nonlinear waves have been obtained for this fluid model, 
see review in \cite{Saut}. In the shallow water limit $h \to 0$, the scaling transformation 
$$
\eta(u,t) := h \tilde{\eta}(\tilde{u},\tilde{t}), \quad \tilde{u} := u + h^{-1}t, \quad \tilde{t} := ht
$$ 
recovers formally the Korteweg--de Vries (KdV) equation
\begin{equation*}
\partial_{\tilde{t}} \tilde{\eta} + \tilde{\eta} \partial_{\tilde{u}} \tilde{\eta} + \frac{1}{3} \partial_{\tilde{u}}^3 \tilde{\eta} = 0 
\end{equation*}
due to the asymptotic expansion 
\begin{equation*}
\mathcal{K}_h = \frac{1}{h} - \frac{1}{3} h \partial_u^2 + \mathcal{O}(h^3).
\end{equation*}
For completeness, in the deep water limit $h \to \infty$, the ILW equation (\ref{ILW}) becomes the Benjamin--Ono (BO) equation
\begin{equation*}
\partial_t \eta + \eta \partial_u \eta + \mathcal{H}(\partial^2_u \eta) = 0,
\end{equation*} 
where $\mathcal{H}$ is the periodic Hilbert transform defined by (\ref{Hilbert}). Both the KdV and BO equations are also integrable by inverse scattering. 

To obtain the local evolution equation (\ref{toy-model}) from the nonlocal model (\ref{single-eq}), we replace $K_h$ given by (\ref{operator-K}) with 
$$
\tilde{K}_h = K_h + \frac{1}{2\pi h} \oint \; \cdot \; du - \frac{1}{h} = \mathcal{K}_h - \frac{1}{h}.
$$
The difference between $\tilde{K}_h$ and $K_h$ appears in the local term $1/h$. It is removed from the mean term in the definition of $K_h$ and from all Fourier modes in the definition of $\tilde{K}_h$. Since $K_h = T_h^{-1} \partial_u$, we can similarly define 
$\tilde{K}_h = \tilde{T}_h^{-1} \partial_u$ and expand asymptotically as $h \to 0$:
$$
\tilde{K}_h = -\frac{1}{3} h \partial_u^2 + \mathcal{O}(h^3), \qquad 
\tilde{T}_h^{-1} = -\frac{1}{3} h \partial_u + \mathcal{O}(h^3).
$$
By using the scaling transformation 
$$
\eta(u,t) := h^{-1} \tilde{\eta}(\tilde{u},\tilde{t}), \qquad 
u := 3^{-1/2} \tilde{u}, \qquad 
t := 3^{-1/2} h^{1/2} \tilde{t}, \qquad 
c = h^{-1/2} \tilde{c}, 
$$
we obtain the formal limit of the model equation (\ref{single-eq}) as $h \to 0$ 
in the form 
\begin{equation}
- 2 \tilde{c} \partial_{\tilde{u}} \partial_{\tilde{t}} \tilde{\eta} = 
\left( -\tilde{c}^2 \partial_{\tilde{u}}^2 - 1 \right) \tilde{\eta} 
+ \tilde{\eta}  \partial_{\tilde{u}}^2 \tilde{\eta} + \frac{1}{2}  \partial_{\tilde{u}}^2 \tilde{\eta}^2.
\label{lin-toy-model}
\end{equation}
Expanding the derivatives, changing the sign, and removing the tilde notations yields the local model (\ref{toy-model}).

\end{document}